\documentclass[a4paper,11pt,reqno]{amsart}
\usepackage[latin1]{inputenc}
\usepackage[english]{babel}
\usepackage{amsmath, amsthm, amssymb, amsopn, amsfonts, amstext, stmaryrd, enumerate, color, mathtools, hyperref, mathrsfs, enumitem, url, comment}
\usepackage[margin=0.75in]{geometry}
\usepackage[final]{showkeys}
\numberwithin{equation}{section}
\usepackage{cite}
\usepackage{tikz}
\usepackage{float}

\hypersetup{
colorlinks=true,
linkcolor=blue,
citecolor=blue,
urlcolor=black,
linktoc=all
}

\newcommand{\D}{\mathbb{D}}
\newcommand{\E}{\mathcal{E}}

\newcommand{\N}{\mathbb{N}}
\newcommand{\NN}{\mathcal{N}}
\newcommand{\R}{\mathbb{R}}

\renewcommand{\S}{\mathbb{S}}

\newcommand{\U}{\mathcal{U}}

\newcommand{\loc}{{\rm loc}}

\newcommand{\Mat}{{\mbox{\normalfont Mat}}}

\newcommand{\PV}{\mbox{\normalfont P.V.}}

\renewcommand{\epsilon} {\varepsilon}

\def\XXint#1#2#3{{\setbox0=\hbox{$#1{#2#3}{\int}$ }
\vcenter{\hbox{$#2#3$ }}\kern-.6\wd0}}

\theoremstyle{plain}
\newtheorem{definition}{Definition}[section]
\newtheorem{theorem}[definition]{Theorem}
\newtheorem{proposition}[definition]{Proposition}
\newtheorem{lemma}[definition]{Lemma}
\newtheorem{corollary}[definition]{Corollary}

\theoremstyle{definition}
\newtheorem{remark}[definition]{Remark}
\newtheorem{example}[definition]{Example}

\renewcommand{\le}{\leqslant}
\renewcommand{\leq}{\leqslant}
\renewcommand{\ge}{\geqslant}
\renewcommand{\geq}{\geqslant}
\newcommand{\e}{\varepsilon}

\title[Non-uniqueness for the nonlocal Liouville equation in $\mathbb{R}$ and applications]{Non-uniqueness for the nonlocal Liouville \\ equation in $\mathbb{R}$ and applications \\ $\, $}

\author[L. Battaglia, M. Cozzi, A. J. Fern\'andez and A. Pistoia]{Luca Battaglia, Matteo Cozzi, Antonio J. Fern\'andez and Angela Pistoia}

\address{
\vspace{0.25cm}
\newline
\textbf{{\small Luca Battaglia}}
\vspace{0.1cm}
\newline \indent Dipartimento di Matematica e Fisica, Universit\`a degli Studi Roma Tre, Largo S. Leonardo Murialdo 1, 00146 Roma (Italy)}
\email{lbattaglia@mat.uniroma3.it}

\address{
\vspace{-0.25cm}
\newline
\textbf{{\small Matteo Cozzi}}
\vspace{0.1cm}
\newline \indent Dipartimento di Matematica ``Federigo Enriques'', Universit\`a degli Studi di Milano, Via Cesare Saldini 50, 20133 Milano (Italy)}
\email{matteo.cozzi@unimi.it}

\address{
\vspace{-0.25cm}
\newline
\textbf{{\small Antonio J. Fern\'andez}} 
\vspace{0.1cm}
\newline \indent Instituto de Ciencias Matem\'aticas, Consejo Superior de Investigaciones Cient\'ificas, Calle Nicolas C\'abrera 13, 28049 Madrid (Spain) 
\newline \indent Departamento de Matem\'aticas, Universidad Aut\'onoma de Madrid, Ciudad Universitaria de Cantoblanco, 28049 Madrid (Spain)}
\email{antonioj.fernandez@uam.es}

\address{
\vspace{-0.25cm}
\newline
\textbf{{\small Angela Pistoia}}
\vspace{0.1cm}
\newline \indent Dipartimento di Scienze di Base e Applicate per l'Ingegneria, Sapienza Universit\`a di Roma, Via Antonio Scarpa 10, 00161 Roma (Italy)}
\email{angela.pistoia@uniroma1.it}

\begin{document}

\nocite{*}

\maketitle

\vspace{-0.3cm}
\begin{abstract}
We construct multiple solutions to the nonlocal Liouville equation
\begin{equation}
\label{eqk} \tag{L}
(-\Delta)^{\frac{1}{2}} u = K(x) e^u \quad \mbox{ in } \R.
\end{equation}
More precisely, for $K$ of the form $K(x) = 1+\e \kappa(x)$ with $\e \in (0,1)$ small and $\kappa \in C^{1,\alpha}(\R) \cap L^{\infty}(\R)$ for some $\alpha > 0$, we prove existence of multiple solutions to \eqref{eqk} bifurcating from the \textit{bubbles}. These solutions provide examples of flat metrics in the half-plane with prescribed geodesic curvature $K(x)$ on its boundary. Furthermore, they imply the existence of multiple ground state soliton solutions for the Calogero-Moser derivative NLS.

\bigbreak

\medbreak
\medbreak
\end{abstract}

\vspace{0.2cm}
\section{Introduction and main results}
\noindent In the present paper we construct multiple solutions to the nonlocal Liouville equation
\begin{equation} \label{maineq}
(-\Delta)^{\frac{1}{2}} u = (1 + \varepsilon \kappa(x)) e^u \quad \mbox{in } \R.
\end{equation}
Here, $\kappa \in C^{1,\alpha}(\R) \cap L^{\infty}(\R)$ for some $\alpha> 0$, $\varepsilon \in (0, 1)$ is a small parameter and $(-\Delta)^{\frac12}$ is the half-Laplace operator, defined for a bounded and sufficiently smooth function $u$ as 
$$
(-\Delta)^{\frac{1}{2}} u(x) := \frac{1}{\pi} \, \PV \int_\R \frac{u(x) - u(y)}{(x - y)^2} \, dy = \frac{1}{\pi} \, \lim_{\delta \rightarrow 0^+} \int_{\R \setminus (x - \delta, x + \delta)} \frac{u(x) - u(y)}{(x - y)^2} \, dy.
$$

The motivation for considering a problem of the form \eqref{maineq} is twofold. Firstly, this problem naturally arises in Riemannian Geometry. Indeed, if we consider the upper half-plane $\R^2_{+}:=\{(x,y) \in \R^2 : y > 0\}$ equipped with the standard Euclidean metric $g_0$ and assume that $u$ is a solution to \eqref{maineq}, the harmonic extension $U$ of $u$ to $\R_{+}^2$ provides a flat conformal metric $g_1 = e^{2U} g_0$ with prescribed geodesic curvature $1+\e\kappa(x)$ on $\R \equiv \partial \R^{2}_{+}$. We refer to the introductions of \cite{BMP21,BMP20,R21,CF22} and the references therein for more details on the geometric context. More recently, problem \eqref{maineq} was linked to the analysis of soliton solutions to the Calogero-Moser derivative NLS
\begin{equation} \tag{CM} \label{CM}
{\rm{i}} \partial_t \psi = - \partial_{xx} \psi + V\psi - (-\Delta)^{\frac12} (|\psi|^2) \psi + \frac14 |\psi|^4 \psi, \quad \textup{ in } [0,T) \times \R.
\end{equation}
Here, $V$ is an external given potential and $\psi$ is complex valued. Equation \eqref{CM} with $V(x) = x^2$ and $V(x) \equiv 0$ was introduced in the physics literature as a formal continuum limit of classical Calogero-Moser systems, see \cite{Physics1,Physics2}. From a rigorous point of view, \eqref{CM} has been very recently analyzed in \cite{GL22}. We refer as well to \cite{AL22} for recent results concerning both \eqref{eqk} and \eqref{CM}. In particular, assuming that $K$ is positive, symmetric-decreasing and satisfies suitable regularity and decay assumptions, the authors prove existence and uniqueness of solution to \eqref{eqk}, as well as some qualitative properties of the solution. In this paper, we intend to establish a sort of counterpart of this, namely that if $K$ changes monotonicity one may have multiple solutions.


For~$\mu > 0$ and~$\xi \in \R$, let us introduce
$$
\U_{\mu, \xi}(x) := \log \left( \frac{2 \mu}{\mu^2 + (x - \xi)^2} \right),
$$
and stress that the functions~$\{ \U_{\mu, \xi} \}_{\mu, \xi}$ are the unique solutions to
\begin{equation} \label{eqforU}
(-\Delta)^{\frac{1}{2}} U = e^U \quad \mbox{in } \R, \qquad \textup{ satisfying} \quad \int_{\R} e^U dx < +\infty.
\end{equation}
We refer to \cite[Theorem 1.8]{DMR15} for a proof (see also \cite{DM17}). Note also that, in \cite[Lemma 3.1]{GL22}, one can find an alternative proof that uses the link between \eqref{eqforU} and \eqref{CM} with $V \equiv 0$. This family of solutions are the so-called \textit{bubbles}. 

Accordingly, we will look for solutions to \eqref{maineq} of the form
\begin{equation} \label{u=U+phi}
u = \U_{\mu, \xi} + \phi,
\end{equation}
with $\phi: \R \to \R$ to be determined and small in a suitable sense.

As we shall see, the existence of such solutions is bound to the properties of the harmonic extension~$\Gamma$ of the function~$\kappa$ to the upper half-plane~$\R^2_+ = \big\{ {(\xi, \mu) : \xi \in \R, \, \mu > 0} \big\}$. Define
\begin{equation} \label{Gamma0def}
\Gamma(\xi, \mu) := \frac{1}{\pi} \int_\R \frac{\kappa(\xi + \mu y)}{1 + y^2} \, dy \quad \mbox{for } \xi \in \R,\ \mu > 0.
\end{equation}
It is well-known that $\Gamma$ is the unique bounded harmonic function in~$\R^2_+$ which agrees with~$\kappa$ on~$\partial \R^2_+ = \big\{ {(\xi, 0) : \xi \in \R} \big\}$---see, e.g.,~\cite{CS07}. Let us also stress that the function~$\Gamma$ is smooth inside~$\R^2_+$. Moreover, we will say that~$(\xi_\star, \mu_\star) \in \R^2_+$ is a non-degenerate critical point for~$\Gamma$ if
$$
\nabla \Gamma(\xi_\star, \mu_\star) = 0 \quad \mbox{and} \quad \det D^2 \Gamma(\xi_\star, \mu_\star) \ne 0.
$$

\noindent Having at hand $\Gamma$, our first main result reads as follows.

\begin{theorem} \label{mainthm}
Let~$(\xi_\star, \mu_\star) \in \R^2_+$ be a non-degenerate critical point for~$\Gamma$. Then, there exist two constants~$\varepsilon_0 \in (0, 1)$ and~$C > 0$ such that the following holds true. For every~$\varepsilon \in (0, \varepsilon_0]$, there exists a solution~$u = u_\varepsilon$ to \eqref{maineq} in the form~\eqref{u=U+phi}, for some~$\xi = \xi_\varepsilon \in \R$ and~$\mu = \mu_\varepsilon \in (0, +\infty)$ such that~$(\xi_\varepsilon, \mu_\varepsilon) \to (\xi_\star, \mu_\star)$ as~$\varepsilon \rightarrow 0^+$ and with~$\phi = \phi_\varepsilon \in L^\infty(\R)$ satisfying~$\| \phi_\varepsilon \|_{L^\infty(\R)} \le C \varepsilon$.
\end{theorem}

Theorem \ref{mainthm} establishes the existence of solutions to \eqref{maineq} provided that non-degenerate critical points for the harmonic extension of $\kappa$ to the upper half-plane do exist. It is natural to wonder whether this assumption is necessary. The next result shows that the condition on~$\Gamma$ is almost necessary for the existence of solutions to~\eqref{maineq} of the form~\eqref{u=U+phi}. Note that we allow here the remainder function~$\phi$ to have logarithmic growth at infinity. In other words, we measure its smallness in the weighted~$L^{\infty}$ norm 
\[ \|g\|_{L_{\log,\xi}^{\infty}(\R)} := \sup_{x \in \R} \, \frac{|g(x)|}{\log(2+|x-\xi|)}, \quad \textup{ for } g \in L^{\infty}_{\rm loc}(\R).\]

\begin{theorem} \label{optimality}
Let~$(\xi,\mu) \in \R^2_+$ and let~$\{u_{\epsilon}\}$ be a sequence of solutions to~\eqref{maineq} of the form~$u_{\epsilon} = \mathcal{U}_{\mu_\epsilon,\,\xi_\epsilon} + \phi_{\epsilon}$ with~$\|\phi_{\epsilon}\|_{L^{\infty}_{\log,\xi}(\R)}  \leq C_0 \, \epsilon^{\alpha}$, for some constants~$C_0 > 0$ and~$\alpha > 1/2$ independent of~$\epsilon$, and with~$(\xi_{\epsilon},\mu_{\epsilon}) \rightarrow (\xi,\mu)$ as~$\epsilon \to 0^{+}$. Then,~$\nabla \Gamma(\xi,\mu) = 0$. 
\end{theorem}


Taking into account the connection established in \cite[Proposition 4.1]{AL22}, as an immediate consequence of Theorem \ref{mainthm}, we get the existence of ground-state soliton solutions to
\begin{equation} \label{CMepsilon}
{\rm{i}} \partial_t \psi = - \partial_{xx} \psi + V_{\epsilon} \psi - (-\Delta)^{\frac12} (|\psi|^2) \psi + \frac14 |\psi|^4 \psi, \quad \textup{ in } [0,T) \times \R,
\end{equation}
with
$$
V_{\epsilon}(x) := \frac14 \frac{(\varepsilon \kappa'(x))^2}{(1+\varepsilon \kappa(x))^2}.
$$
More precisely, we have the following. 
\begin{corollary} \label{corCMepsilon}
Let~$(\xi_\star, \mu_\star) \in \R^2_+$ be a non-degenerate critical point for~$\Gamma$. Then, there exist two constants~$\varepsilon_0 \in (0, 1)$ and~$C > 0$ such that the following holds true. For every~$\varepsilon \in (0, \varepsilon_0]$, there exists a ground-state soliton solution $\psi = \psi_\varepsilon$ to \eqref{CMepsilon} in the form
\begin{equation} \label{psiepsilon}
\psi_{\varepsilon}(x,t) = \sqrt{(1+\varepsilon \kappa(x)) e^{\mathcal{U}_{\mu_{\varepsilon}, \xi_{\varepsilon}}(x) + \phi_{\epsilon}(x)}}\,,
\end{equation}
for some~$\xi = \xi_\varepsilon \in \R$ and~$\mu = \mu_\varepsilon \in (0, +\infty)$ such that~$(\xi_\varepsilon, \mu_\varepsilon) \to (\xi_\star, \mu_\star)$ as~$\varepsilon \rightarrow 0^+$ and with~$\phi = \phi_\varepsilon \in L^\infty(\R)$ satisfying~$\| \phi_\varepsilon \|_{L^\infty(\R)} \le C \varepsilon$.
\end{corollary}

In virtue of Theorem \ref{mainthm}, our analysis is reduced to find assumptions on $\kappa$ implying the existence of (multiple) non-degenerate critical points for $\Gamma$. We analyze~$\Gamma$ in detail in Sections \ref{section4} and \ref{section5}. However, we would like to illustrate here what kind of multiplicity results we are able to prove. To that end, we need to introduce the weighted $C^2$ space
\begin{equation}\label{c*}
C^{2}_{*,\beta}(\mathbb R):=\left\{\kappa\in C^{2}(\mathbb R)\ :\ \|\kappa\|_{C^{2}_{*,\beta}(\R)} <+\infty \right\},
\end{equation}
equipped with the norm
\begin{equation} \label{c*norm}
\|\kappa\|_{C^{2}_{*,\beta}(\R)}:=\sup\limits_{x\in\mathbb R} |\kappa(x)| + \sup\limits_{x\in\mathbb R} \left(1+|x|^{2+\beta}\right)|\kappa'(x)| + \sup\limits_{x\in\mathbb R}  |\kappa''(x)|,
\end{equation}
for~$\beta \in (0, 1)$. We then assume that $\kappa \in C^2_{*,\beta}(\R)$ and that furthermore satisfies:
\begin{align}
\circ\ & \, \textup{There exists } R_0 > 0 \textup{ such that } x\kappa'(x) < 0 \textup{ for all }|x| > R_0; \label{kprimneg} \\
\circ\ & \, \int_{\R} \kappa'(x) x \, dx < 0. \label{kintpos}
\end{align}
 Under these stronger assumptions on~$\kappa$ we (in particular) prove the following.  

\begin{theorem}\label{main2}
Assume that~$\kappa_0\in C^{2}_{*,\beta}(\mathbb R)$ for some~$\beta > 0$ satisfies~\eqref{kprimneg} and~\eqref{kintpos}. Given~$N \in \N \cup \{ 0 \}$, suppose that~$\kappa_0$ has exactly~$N+1$ global maxima and~$N$ global minima, that these extrema are all non-degenerate, and that~$\kappa_0$ has no other stationary point. Then, there exists a sequence~$\{ \kappa_j \} \subset C^{2}_{*,\beta}(\mathbb R)$ converging to~$\kappa_0$ in~$C^{2}_{*,\beta}(\mathbb R)$ for which the following holds true. For each~$j \in \N$, there exists~$\varepsilon_j \in (0, 1)$ such that equation~\eqref{maineq} (with~$\kappa = \kappa_j$) has at least~$N$ solutions of the form~\eqref{u=U+phi} for all~$\varepsilon \in (0, \varepsilon_j]$.
\end{theorem}

Arguing as for Corollary~\ref{corCMepsilon} one can easily obtain a multiplicity result concerning~\eqref{CMepsilon}. To state such as result, we need some notation. We define
$$
V_{j,\epsilon}(x):= \frac14 \frac{(\varepsilon \kappa_j'(x))^2}{(1+\varepsilon \kappa_j(x))^2}, \quad \textup{ for all } j \in \N,\ \epsilon \in (0,1) \textup{ and } \kappa_j \in C^2_{*,\beta}(\R).
$$

\begin{corollary} \label{corMepsilon2}
Assume that~$\kappa_0\in C^{2}_{*,\beta}(\mathbb R)$ for some~$\beta > 0$ satisfies~\eqref{kprimneg} and~\eqref{kintpos}. Given~$N \in \N \cup \{ 0 \}$, suppose that~$\kappa_0$ has exactly~$N+1$ global maxima and~$N$ global minima, that these extrema are all non-degenerate, and that~$\kappa_0$ has no other stationary point. Then, there exists a sequence~$\{ \kappa_j \} \subset C^{2}_{*,\beta}(\mathbb R)$ converging to~$\kappa_0$ in~$C^{2}_{*,\beta}(\mathbb R)$ for which the following holds true. For each~$j \in \N$, there exists~$\varepsilon_j \in (0, 1)$ such that \eqref{CMepsilon} (with $V_\epsilon = V_{j,\epsilon}$) has at least $N$ solutions of the form \eqref{psiepsilon} for all $\epsilon \in (0,\epsilon_j]$.
\end{corollary}


\begin{remark} $ $
\begin{enumerate}[label=$(\alph*)$,leftmargin=*]
\item If $\kappa \in  C_{*,\beta}^2(\R)$ for some $\beta > 0$ and
$
\sup_{x \in \R} \left(1+|x|^{1+\beta}\right)|\kappa(x)| < + \infty,
$
then \eqref{kintpos} is equivalent to the somewhat more natural assumption
$$
\int_{\R} \kappa(x)\, dx > 0.
$$
\item Naturally, in Theorem \ref{main2} and Corollary \ref{corMepsilon2}, it might happen that~$\kappa_0$ itself gives rise to~$N$ solutions---i.e., that~$\kappa_j = \kappa_0$ for all~$j \in \N$. The realization of this possibility is connected with the harmonic extension~$\Gamma_0$ of~$\kappa_0$ (as defined in~\eqref{Gamma0def}) having non-degenerate critical points. See~Section~\ref{section4} and in particular the proof of Theorem~\ref{main2} included there for more details.
\end{enumerate}
\end{remark}

The proof of Theorem \ref{mainthm} is perturbative and uses a Lyapunov-Schmidt reduction method. This kind of approach was already used by Ambrosetti, Garc\'ia Azorero \& Peral \cite{AAP99} and Grossi \& Prashanth \cite{GP05} to deal with the higher dimensional counterparts of \eqref{maineq}. However, our approach is closer to the one recently used in \cite{CF22} by the second and third authors. Actually, despite that both the problems under consideration and the kind of results obtained are different, our analysis here and the one performed in \cite{CF22} share some technical elements. 

As already indicated, having at hand Theorem \ref{mainthm}, we are led to analyze the set of critical points for the harmonic extension of $\kappa$ to the upper half-plane, namely the set of critical points for the function $\Gamma$ given in \eqref{Gamma0def}. Roughly speaking, we want to count the number of critical points of $\Gamma$ using only assumptions on $\kappa$. This is a rather classical problem. Indeed, using a conformal map, we could reformulate such problem as follows: given $g \in L^{\infty}(\mathbb{S}^1)$, let $u$ be the unique solution to
$$
\left\{
\begin{aligned}
& \Delta u  = 0, \quad && \textup{ in } \mathbb{D}, \\
& u  = g, && \textup{ on } \mathbb{S}^1.
\end{aligned}
\right.
$$
We would like to count the number of critical points of~$u$ in $\mathbb{D}$ using only assumptions on~$g$. Despite its apparent simplicity, the answer to this problem is far from being complete. In Theorem~\ref{thmdeg} and Proposition~\ref{propexactnum} we obtain some partial answers using a topological degree approach. However, as highlighted by some of the examples in Section \ref{section5}, the general picture is far from being clear.  

\subsection*{Organization of the paper} The remaining of the paper is organized as follows. Section \ref{thm1.1proofsec} is devoted to the proof of Theorem \ref{mainthm}. More precisely, in Subsection \ref{subsectLT}, we develop the required linear theory for the linearized operator at $\mathcal{U}_{\mu,\xi}$; in Subsection \ref{subsectNLT}, we use these results to address a projected version of the original nonlinear problem; and, in Subsection \ref{subsectFDR}, we deal with the finite-dimensional reduced problem concluding thus the proof of Theorem \ref{mainthm}. Section \ref{sect3} is devoted to show that the condition on $\Gamma$ in Theorem \ref{mainthm} is almost necessary, namely to prove Theorem \ref{optimality}. In Section \ref{section4} we analyze the function $\Gamma$ introduced in \eqref{Gamma0def} and prove a few results of independent interest concerning $\Gamma$---including a genericity result, see Lemma \ref{denslem}. Those results lead to the proof of Theorem \ref{main2}. In Section \ref{section5} we provide some explicit examples of functions $\kappa$ such that \eqref{maineq} has or does not have solutions of the form \eqref{u=U+phi}. We also make some comments and remarks comparing our results with previous ones in the literature. The paper is closed by an appendix where we provide an alternative proof of a statement contained in Lemma \ref{denslem}.

\subsection*{Acknowledgments}

This work has received funding from the European Research Council (ERC) under the European Union's Horizon 2020 research and innovation program through the Consolidator Grand agreement 862342 (A.J.F.).  L.B., M.C., and A.P. have been partially supported by GNAMPA (Italy) as part of INdAM. M.C. has also been supported by the Spanish grant PID2021-123903NB-I00 funded by MCIN/AEI/10.13039/501100011033 and by ERDF ``A way of making Europe''. Part of this work has been done while A.J.F. was visiting the Universit\`a di Roma ``La Sapienza''. He thanks his hosts for the kind hospitality and the financial support.

\section{Proof of Theorem~\ref{mainthm} and Corollary~\ref{corCMepsilon}} \label{thm1.1proofsec} 

\noindent
It is immediate to see that~$u$ of the form~\eqref{u=U+phi} solves~\eqref{maineq} if and only if~$\phi$ is a solution of
\begin{equation} \label{maineqforphi}
L_{\mu, \xi} \, \phi = - \E + \NN(\phi) \quad \mbox{in } \R,
\end{equation}
where
\begin{align*}
L_{\mu, \xi} \, \phi & := (-\Delta)^{\frac{1}{2}} \phi - e^{\U_{\mu, \xi}} \phi, \\
\E = \E_{\mu, \xi} & := (-\Delta)^{\frac{1}{2}} \U_{\mu, \xi} - (1 + \varepsilon \kappa(x)) e^{\U_{\mu, \xi}}, \\
\NN(\phi) = \NN_{\mu, \xi}[\phi] & := e^{\U_{\mu, \xi}} \left\{ \left( e^\phi - 1 - \phi \right) + \varepsilon \kappa(x) \left( e^\phi - 1 \right) \right\}.
\end{align*}

As it turns out, the smallness of~$\phi$ will be evaluated in~$L^\infty(\R)$. On the other hand, to measure the error term~$\E$ and the nonlinear term~$\NN$ we will make use of the following weighted~$L^\infty$ norm, defined on a function~$g \in L^\infty_\loc(\R)$ by
\begin{equation} \label{Linftysigmanormdef}
\| g \|_{L^\infty_{\sigma, \xi}(\R)} := \sup_{x \in \R} \, (1 + |x - \xi|)^{1 + \sigma} |g(x)|,
\end{equation}
with~$\sigma \in (0, 1)$ fixed. We have the following estimate of the size of~$\E$ and~$\NN$.

\begin{lemma} \label{ENboundslem}
Given~$\mu_0 \ge 1$, there exist two constants~$C_1 > 0$, depending only on~$\mu_0$ and~$\sigma$, and~$C_2 > 0$, also depending on~$\| \kappa \|_{L^\infty(\R)}$, such that, for any~$\xi \in \R$ and~$\mu \in \left[ \frac{1}{\mu_0}, \mu_0 \right]$,
\begin{equation} \label{Ebound}
\left\| \E \right\|_{L^\infty_{\sigma, \xi}(\R)} \le C_1 \varepsilon
\end{equation}
and
\begin{equation} \label{Nbound}
\left\| \NN(\phi) \right\|_{L^\infty_{\sigma, \xi}(\R)} \le C_2 \| \phi \|_{L^\infty(\R)} \left( \| \phi \|_{L^\infty(\R)} + \varepsilon \right),
\end{equation}
for all~$\phi \in L^\infty(\R)$ such that~$\| \phi \|_{L^\infty(\R)} \le 1$. Moreover,
\begin{equation} \label{NLipbound}
\left\| \NN(\phi) - \NN(\psi) \right\|_{L^\infty_{\sigma, \xi}(\R)} \le C_2 \| \phi - \psi \|_{L^\infty(\R)} \left( \| \phi \|_{L^\infty(\R)} + \| \psi \|_{L^\infty(\R)} + \varepsilon \right),
\end{equation}
for all~$\phi, \psi \in L^\infty(\R)$ such that~$\| \phi \|_{L^\infty(\R)}, \| \psi \|_{L^\infty(\R)} \le 1$.
\end{lemma}
\begin{proof}
As~$\U_{\mu, \xi}$ solves~\eqref{eqforU}, we see that
\begin{equation} \label{Ealternateform}
\E(x) = \frac{2 \mu \varepsilon \kappa(x)}{\mu^2 + (x - \xi)^2},
\end{equation}
for all~$x \in \R$. From this identity, estimate~\eqref{Ebound} follows at once. To obtain~\eqref{Nbound}, we observe that~$|e^\phi - 1| \le e^{|\phi|} |\phi|$ and~$|e^\phi - 1 - \phi| \le \frac{e^{|\phi|}}{2} |\phi|^2$ for every~$\phi \in \R$. Hence,
\begin{equation} \label{estimateNNphi}
|\NN(\phi)(x)| \le \frac{2 \mu}{\mu^2 + (x - \xi)^2} \bigg( \frac{e^{|\phi(x)|}}{2} \, |\phi(x)|^2 + \varepsilon \| \kappa \|_{L^\infty(\R)} e^{|\phi(x)|} |\phi(x)|  \bigg).
\end{equation}
In particular, if $\|\phi\|_{L^{\infty}(\R)} \leq 1$, we get 
$$
|\NN(\phi)(x)| \le \frac{8 \mu}{\mu^2 + (x - \xi)^2} \Big(  |\phi(x)|^2 + \varepsilon \| \kappa \|_{L^\infty(\R)} |\phi(x)|  \Big).
$$
which gives~\eqref{Nbound}. Estimate~\eqref{NLipbound} is analogous.
\end{proof}

\subsection{Linear theory} \label{subsectLT}

We study in this subsection the linearized operator~$L_{\mu, \xi}$. We begin by analyzing its kernel. Since~$\U_{\mu, \xi}$ is a solution of~\eqref{eqforU}, it is clear that both derivatives
\begin{equation} \label{Zdefs}
Z_{0, \mu, \xi}(x) := \partial_\mu \U_{\mu, \xi}(x) = \frac{1}{\mu} - \frac{2 \mu}{\mu^2 + (x - \xi)^2}, \quad Z_{1, \mu, \xi}(x) := \partial_\xi \U_{\mu, \xi}(x) = \frac{2 (x - \xi)}{\mu^2 + (x - \xi)^2}
\end{equation}
annihilate~$L_{\mu, \xi}$. We recall the following known non-degeneracy result, showing that these two functions actually span the bounded kernel of~$L_{\mu, \xi}$.

\begin{lemma}[\hspace{-0.01cm}\cite{DdM05,S19,CF22}] \label{nondeglem}
If~$Z \in L^\infty(\R)$ is a solution of~$L_{\mu, \xi} \, Z = 0$ in~$\R$, then~$Z$ is a linear combination of~$Z_{0, \mu, \xi}$ and~$Z_{1, \mu, \xi}$.
\end{lemma}

Let~$\bar{R} > 1$ be fixed and~$\chi \in C^\infty(\R)$ be an even cutoff function satisfying~$0 \le \chi \le 1$ in~$\R$,~$\chi = 1$ in~$[- \bar{R}, \bar{R}]$, and~$\chi = 0$ in~$\R \setminus (- \bar{R} - 1, \bar{R} + 1)$. We set~$\chi_\xi := \chi(\cdot - \xi)$. Also, we fix two numbers $\mu_0 \geq 1$ and $ \sigma \in (0,1)$. Finally, we consider the space of functions defined by
$$
L^\infty_{\sigma, \xi}(\R) := \Big\{ {g \in L^\infty_\loc(\R) : \| g \|_{L^\infty_{\sigma, \xi}(\R)} < +\infty} \Big\},
$$
with~$\| \cdot \|_{L^\infty_{\sigma, \xi}(\R)}$ as in~\eqref{Linftysigmanormdef}. The rest of this subsection is devoted to prove the following. 

\begin{proposition} \label{mainlinearprop}
Given any~$g \in L^\infty_{\sigma, \xi}(\R)$, there exists a unique bounded weak solution~$\phi$ of
\begin{equation} \label{projeq}
L_{\mu, \xi} \, \phi = g + \sum_{i = 0}^1 d_i \chi_\xi Z_{i, \mu, \xi} \quad \mbox{in } \R,
\end{equation}
for some uniquely determined~$d_0, d_1 \in \R$, satisfying the orthogonality conditions
\begin{equation} \label{ortcond}
\int_{\R} \phi \chi_\xi Z_{i, \mu, \xi} \, dx = 0 \quad \mbox{for } i = 0, 1.
\end{equation}
Moreover,
\begin{equation} \label{linearLinftyest}
\| \phi \|_{L^\infty(\R)} \le C_0 \| g \|_{L^\infty_{\sigma, \xi}(\R)},
\end{equation}
for some constant~$C_0 > 0$ depending only on~$\mu_0$,~$\sigma$, and~$\bar{R}$.
\end{proposition}

Clearly, after a translation we may reduce to the case~$\xi = 0$. In the remaining of the subsection we will always assume this to hold.

To prove Proposition~\ref{mainlinearprop}, we begin with the following a priori estimate.

\begin{lemma} \label{apriorilem}
There exist two constants~$R_0 > 0$ and~$C \ge 1$, depending only on~$\mu_0$,~$\sigma$, and~$\bar{R}$ such that the following holds true. If~$R \ge R_0$ and~$\phi \in L^\infty(\R)$ is a weak solution of
$$
\begin{cases}
L_{\mu, 0} \, \phi = g & \quad \mbox{in } (- R, R), \\
\phi = 0 & \quad \mbox{in } \R \setminus (- R, R),
\end{cases}
$$
satisfying
\begin{equation} \label{ortcondxi=0}
\int_\R \phi \chi Z_{i, \mu, 0} \, dx = 0 \quad \mbox{for } i = 0, 1,
\end{equation}
with~$\mu \in \left[ \frac{1}{\mu_0}, \mu_0 \right]$ and~$g \in L^\infty_{\sigma, 0}(\R)$, then it holds
$$
\| \phi \|_{L^\infty(\R)} \le C \| g \|_{L^\infty_{\sigma, 0}(\R)}.
$$
\end{lemma}
\begin{proof}
We argue by contradiction and assume that there exists a sequence~$\{ R_k \}$ of positive numbers such that~$R_k \rightarrow +\infty$ as~$k \rightarrow +\infty$ and for which the following holds. For every~$k \in \N$, there exists~$\mu_k \in \left[\frac{1}{\mu_0}, \mu_0 \right]$, a function~$g_k \in L^\infty_{\sigma, 0}(\R)$, and a bounded solution~$\phi_k$ of
\begin{equation} \label{systforphik}
\begin{cases}
L_{\mu_k, 0} \, \phi_k = g_k & \quad \mbox{in } (- R_k, R_k), \\
\phi_k = 0 & \quad \mbox{in } \R \setminus (- R_k, R_k),
\end{cases}
\end{equation}
satisfying
\begin{equation} \label{ortcondforphik}
\int_{\R} \phi_k \chi Z_{i, \mu_k, 0} \, dx = 0 \quad \mbox{for } i = 0, 1,
\end{equation}
but for which
$$
\| \phi_k \|_{L^\infty(\R)} > k \| g_k \|_{L^\infty_{\sigma, 0}(\R)}.
$$
By the linearity of the problem, we may further assume that
\begin{equation} \label{phikgksizes}
\| \phi_k \|_{L^\infty(\R)} = 1 \quad \mbox{and} \quad \| g_k \|_{L^\infty_{\sigma, 0}(\R)} \le \frac{1}{k}.
\end{equation}

Let~$\rho \ge 1$ be fixed. We claim that
\begin{equation} \label{claimbarrier}
\| \phi_k \|_{L^\infty(\R)} \le C \left( \| \phi_k \|_{L^\infty \left( (-\rho, \rho) \right)} + \| g_k \|_{L^\infty_{\sigma, 0}(\R)} \right),
\end{equation}
for some constant~$C > 0$ independent of~$k$ and provided~$\rho$ is sufficiently large but also independent of~$k$. This follows from a barrier argument. Consider the function~$w_\sigma(x) := (1 + x^2)^{- \frac{\sigma}{2}}$ for~$x \in \R$. According to~\cite[Lemma~3.4]{CF22}, there exist two constants~$c_0 \in (0, 1)$ and~$\rho_0 \ge 1$, depending only on~$\sigma$, such that
$$
(-\Delta)^{\frac{1}{2}} w_\sigma(x) \le - \frac{c_0}{|x|^{1 + \sigma}} \quad \mbox{for all } x \in \R \setminus [-\rho_0, \rho_0].
$$
From this, it is easy to see that the function
$$
\overline{\phi}_k := \left( \| \phi_k \|_{L^\infty \left( (-\rho, \rho) \right)} + \frac{2}{c_0} \, \| g_k \|_{L^\infty_{\sigma, 0}(\R)} \right) \left( 2 - w_\sigma \right)
$$
satisfies
$$
\begin{cases}
L_{\mu_k, 0} \, \overline{\phi}_k \ge g_k & \quad \mbox{in } (- R_k, R_k) \setminus [-\rho, \rho], \\
\overline{\phi}_k \ge \phi_k & \quad \mbox{in } [- \rho, \rho] \cup \big( {\R \setminus (-R_k, R_k)} \big),
\end{cases}
$$
provided~$\rho \ge \rho_0$ is sufficiently large, in dependence of~$\mu_0$ and~$\sigma$ only. Since the operator~$L_{\mu_k, 0}$ enjoys the maximum principle outside of a sufficiently large neighborhood of the origin (see~\cite[Lemma~3.5]{CF22}), up to taking~$\rho$ even larger we conclude that~$\phi_k \le \overline{\phi}_k$ in the whole~$\R$. Since an analogous argument could be made to obtain a lower bound for~$\phi_k$ (using~$- \overline{\phi}_k$ as subsolution), we conclude that
$$
|\phi_k(x)| \le \overline{\phi}_k(x) \le \frac{4}{c_0} \left( \| \phi_k \|_{L^\infty \left( (-\rho, \rho) \right)} + \| g_k \|_{L^\infty_{\sigma, 0}(\R)} \right) \quad \mbox{for all } x \in \R
$$
which is~\eqref{claimbarrier}.

From~\eqref{claimbarrier} and~\eqref{phikgksizes}, we deduce that
$$
\| \phi_k \|_{L^\infty \left( (-\rho, \rho) \right)} \ge 2 \delta,
$$
for all~$k$'s sufficiently large and for some~$\delta > 0$ independent of~$k$. Hence, for any large~$k$ there exists a point~$p_k \in (- \rho, \rho)$ at which~$|\phi_k(p_k)| \ge \delta$. Notice that, by elliptic regularity, the~$\phi_k$'s are continuous functions satisfying, for any fixed~$r > 0$,
$$
\| \phi_k \|_{C^\alpha([- r, r])} \le C
$$
for any~$\alpha \in (0, 1)$ and some constant~$C > 0$ independent of~$k$ and for~$k$ sufficiently large (in particular, large enough to have that~$R_k \ge r + 1$). By compactness and a diagonal argument, we conclude that, up to a subsequence,~$\{ \phi_k \}$ converges to a function~$\phi_\infty$ in~$C^{\alpha'}_\loc(\R)$, for any~$\alpha' \in (0, 1)$. Up to taking a further subsequence, we clearly also have that~$p_k \rightarrow p_\infty \in [- \rho, \rho]$ and~$\mu_k \rightarrow \mu_\infty \in \left[ \frac{1}{\mu_0}, \mu_0 \right]$. Recalling~\eqref{systforphik},~\eqref{ortcondforphik}, and~\eqref{phikgksizes}, we see that~$\phi_\infty \in L^\infty(\R)$ is such that
$$
\begin{dcases}
L_{\mu_\infty, 0} \, \phi_\infty = 0 & \quad \mbox{in } \R, \\
\int_{\R} \phi_\infty \chi Z_{i, \mu_\infty, 0} \, dx = 0 & \quad \mbox{for } i = 0, 1, \\
|\phi_\infty(p)| \ge \delta.
\end{dcases}
$$
By Lemma~\ref{nondeglem}, we conclude that
$$
\phi_\infty = A_0 Z_{0, \mu_\infty, 0} + A_1 Z_{1, \mu_\infty, 0} \quad \mbox{in } \R,
$$
for some constants~$A_0, A_1 \in \R$. But then, the orthogonality conditions yield~$A_0 = A_1 = 0$ and thus that~$\phi_\infty = 0$ in~$\R$, contradicting the fact that~$|\phi_\infty(p)| \ge \delta$. The proof is complete.
\end{proof}

Thanks to the previous a priori estimate, we may now solve the Dirichlet problem in bounded intervals. We begin with the following lemma, which provides to useful auxiliary functions modeled upon~$Z_{0, \mu, 0}$ and~$Z_{1, \mu, 0}$.

\begin{lemma} \label{2funclem}
Given~$\mu_0 \ge 1$, there exist two constants~$R_1 > 2 \mu_0$ and~$C \ge 1$, depending only on~$\mu_0$, such that the following holds true. For every~$r \ge R_1$ and~$\mu \in \left[ \frac{1}{\mu_0}, \mu_0 \right]$, there exist two functions~$\tilde{z}_0, \tilde{z}_1 \in C^\infty(\R)$ supported in~$[- 2 r, 2 r]$, with~$\tilde{z}_0$ even and~$\tilde{z}_1$ odd, satisfying
\begin{align*}
\tilde{z}_0 & = Z_{0, \mu, 0} \,\, \mbox{ in } [- 2 \mu_0, 2 \mu_0], && 0 \le \tilde{z}_0 \le Z_{0, \mu, 0} \,\, \mbox{ in } [-2 r, 2 r] \setminus [- 2 \mu_0, 2 \mu_0], \\
\tilde{z}_1 & = Z_{1, \mu, 0} \,\, \mbox{ in } [- r, r], && 0 \le \tilde{z}_1 \le Z_{1, \mu, 0} \,\, \mbox{ in } [r,2r],
\end{align*}
and
\begin{align*}
|L_{\mu, 0} \, \tilde{z}_0(x)| & \le \frac{C}{\log r} \left( \frac{\log (2 + |x|)}{(1 + |x|)^2} + \frac{r}{(r + |x|)^2} \right) & \mbox{for all } x \in \R, \\
|L_{\mu, 0} \, \tilde{z}_1(x)| & \le \frac{C}{(1 + |x|)(r + |x|)} & \mbox{for all } x \in \R.
\end{align*}
Moreover,~$\tilde{z}_0 \rightarrow Z_{0, \mu, 0}$ and~$\tilde{z}_{1} \rightarrow Z_{1, \mu, 0}$ uniformly on compact sets as~$r \rightarrow +\infty$.
\end{lemma}
\begin{proof}
The two functions were constructed in~\cite[Section~3]{CF22}.

Specifically, we can take~$\tilde{z}_1$ as the function of~\cite[Lemma~3.8]{CF22} corresponding to~$\delta_0 = \frac{1}{2}$ and~$\varepsilon = \frac{1}{12 r}$. The estimate for~$L_{\mu, 0} \, \tilde{z}_1$ is contained in the proof of said result, namely equations~(3.40),~(3.41), and the unnumbered one at the very end of the proof.

The function~$\tilde{z}_0$ can be obtained from~\cite[Lemma~3.6]{CF22}, with~$\delta_0 = \frac{1}{2}$,~$\varepsilon = \frac{1}{12 r}$, and~$R = 2 \mu_0$. The bound for~$L_{\mu, 0} \, \tilde{z}_0$ does not follow from~(3.19) there, but can be deduced from its proof, in particular by recalling formulas~(3.22),~(3.23) and carefully analyzing the subsequent estimates, up to~(3.26). Note that the requirement that~$R$ must be sufficiently large made in the statement of~\cite[Lemma~3.6]{CF22} is only required to get~(3.20), not~(3.19).

The conclusion on the locally uniform convergence to the~$Z_{i, \mu, 0}$'s is obvious for~$\tilde{z}_1$, since it agrees with~$Z_{1, \mu, 0}$ on~$[-r, r]$. That of~$\tilde{z}_0$ to~$Z_{0, \mu, 0}$ can be inferred from the proof of~\cite[Lemma~3.6]{CF22}, specifically by looking at the definition of~$\widetilde{Z}_0$ and at the bound on~$|h(Y) - 1|$.
\end{proof}

\begin{lemma} \label{existinboundedlem}
There exist two constants~$R_0 > 0$ and~$C \ge 1$, depending only on~$\mu_0$,~$\sigma$, and~$\bar{R}$, such that the following holds true. For every~$R \ge R_0$ and~$g \in L^\infty_{\sigma, 0}(\R)$, there exists a unique weak solution~$\phi$ to the problem
\begin{equation} \label{projprob}
\begin{dcases}
L_{\mu, 0} \, \phi = g + \sum_{i = 0}^1  d_i \chi Z_{i, \mu, 0} & \quad \mbox{in } (- R, R), \\
\phi = 0 & \quad \mbox{in } \R \setminus (- R, R),
\end{dcases}
\end{equation}
for some uniquely determined~$d_0, d_1 \in \R$, which satisfies the orthogonality conditions~\eqref{ortcondxi=0}. Furthermore, it holds
\begin{equation} \label{Linftyestforprojprob}
\| \phi \|_{L^\infty(\R)} + |d_0| + |d_1| \le C \| g \|_{L^\infty_{\sigma, 0}(\R)}.
\end{equation}
\end{lemma}
\begin{proof}
We begin by establishing the~$L^\infty$ estimate~\eqref{Linftyestforprojprob}. Let~$\phi$ be a weak solution of~\eqref{projprob} satisfying~\eqref{ortcond}, for some~$d_0, d_1 \in \R$. By elliptic estimates,~$\phi$ is continuous up to the boundary and, in particular, bounded. Hence, Lemma~\ref{apriorilem} yields the estimate
\begin{equation} \label{phiestintermsofdis}
\| \phi \|_{L^\infty(\R)} \le C \left( \| g \|_{L^\infty_{\sigma, 0}(\R)} + |d_0| + |d_1| \right),
\end{equation}
for some constant~$C \ge 1$ depending only on~$\mu_0$,~$\sigma$,~$\bar{R}$ and provided~$R$ is sufficiently large, also in dependence of these quantities alone. To control the constants~$|d_0|$ and~$|d_1|$, we take advantage of the two functions obtained in Lemma~\ref{2funclem}. By testing the equation against the function~$\tilde{z}_1$, corresponding to~$r = R - 1$ and taking~$R$ sufficiently large, we get that
$$
d_1 \int_\R \chi Z_{1, \mu, 0}^2 \, dx = \int_{\R} \Big( {\phi L_{\mu, 0} \, \tilde{z}_1 - g \tilde{z}_1} \Big) \, dx,
$$
where we used that~$\tilde{z}_1 = Z_{1, \mu, 0}$ on the support of~$\chi$. From this and the properties of~$\tilde{z}_1$, we easily find that
\begin{align*}
|d_1| & \le C \left( \| \phi \|_{L^\infty(\R)} \int_{\R} |L_{\mu, 0} \, \tilde{z}_1(x)| \, dx + \| g \|_{L^\infty_{\sigma, 0}(\R)} \int_{\R} \frac{|\tilde{z}_1(x)|}{(1 + |x|)^{1 + \sigma}} \, dx \right) \\
& \le C \left( \| \phi \|_{L^\infty(\R)} \int_{\R} \frac{dx}{(1 + |x|) (R + |x|)} + \| g \|_{L^\infty_{\sigma, 0}(\R)} \int_{\R} \frac{dx}{(1 + |x|)^{1 + \sigma}} \, dx \right) \\
& \le C \left( R^{- \tau} \| \phi \|_{L^\infty(\R)} + \| g \|_{L^\infty_{\sigma, 0}(\R)} \right).
\end{align*}
for any fixed~$\tau \in (0, 1)$ and with~$C > 0$ depending only on~$\mu_0$,~$\bar{R}$,~$\sigma$, and~$\tau$. Testing the equation with~$\tilde{z}_0$ instead, we obtain that
\begin{align*}
|d_0| & \le C \left( \| \phi \|_{L^\infty(\R)} \int_{\R} |L_{\mu, 0} \, \tilde{z}_0(x)| \, dx + \| g \|_{L^\infty_{\sigma, 0}(\R)} \int_{\R} \frac{|\tilde{z}_0(x)|}{(1 + |x|)^{1 + \sigma}} \, dx \right) \\
& \le C \left( \frac{\| \phi \|_{L^\infty(\R)}}{\log R} \int_{\R} \left( \frac{\log (2 + |x|)}{(1 + |x|)^2} + \frac{R}{(R + |x|)^2} \right) dx + \| g \|_{L^\infty_{\sigma, 0}(\R)} \int_{\R} \frac{dx}{(1 + |x|)^{1 + \sigma}} \right) \\
& \le  C \left( \frac{\| \phi \|_{L^\infty(\R)}}{\log R} + \| g \|_{L^\infty_{\sigma, 0}(\R)} \right).
\end{align*}
By combining these two estimates with~\eqref{phiestintermsofdis} and taking~$R$ sufficiently large to reabsorb the~$L^\infty$ norm of~$\phi$ to the left-hand side, we are led to~\eqref{Linftyestforprojprob}. 

Note that~\eqref{Linftyestforprojprob} yields in particular the uniqueness of a triple~$(\phi, d_0, d_1)$ for a given~$g \in L^\infty_{\sigma, 0}(\R)$. We now address its existence. Let
$$
H := \left\{ \phi \in H^{\frac{1}{2}}(\R) : \phi = 0 \mbox{ in } \R \setminus (-R, R) \mbox{ and } \int_{\R} \phi \chi Z_{i, \mu, 0} \, dx = 0 \mbox{ for } i = 0, 1 \right\}.
$$
Thanks to the fractional Poincar\'e inequality,~$H$ is a Hilbert space with inner product
$$
\langle \phi, \psi \rangle_H := \frac{1}{2 \pi} \int_{\R} \int_{\R} \frac{\big( {\phi(x) - \phi(y)} \big) \big( {\psi(x) - \psi(y)} \big)}{(x - y)^2} \, dx dy.
$$
It is then not hard to see that~$\phi$ is a weak solution of problem~\eqref{projprob} for some~$d_0, d_1 \in \R$ and fulfilling the orthogonality conditions~\eqref{ortcond} if and only if~$\phi \in H$ and it satisfies
$$
\langle \phi, \psi \rangle_H - \int_{-R}^R e^{\U_{\mu, 0}} \phi \psi \, dx = \int_{-R}^R g \psi \, dx \quad \mbox{for all } \psi \in H.
$$
Thanks to estimate~\eqref{Linftyestforprojprob} (applied with~$g = 0$), one can easily show via Fredholm's alternative that this problem has a unique solution~$\phi \in H$ for all~$g \in L^\infty_{\sigma, 0}(\R)$. This concludes the proof.
\end{proof}

By letting~$R \rightarrow +\infty$, we obtain the unique solvability of the problem in the whole real line, namely we establish Proposition~\ref{mainlinearprop}.

\begin{proof}[Proof of Proposition~\ref{mainlinearprop}]
According to Lemma~\ref{existinboundedlem}, let~$(\phi^k, d_0^k, d_1^k)$ be the unique solution of problem~\eqref{projprob} for~$R = k \in \N$ sufficiently large, with~$\phi^k$ satisfying the orthogonality conditions~\eqref{ortcond}. Thanks to the bound~\eqref{Linftyestforprojprob} and standard regularity estimates for the half-Laplacian, we see that, up to a subsequence,~$\{ d_0^k \}_k$ and~$\{ d_1^k \}_k$ respectively converge to two constants~$d_0, d_1 \in \R$ and~$\phi^k \rightarrow \phi$ in~$C^\alpha_\loc(\R)$ for any~$\alpha \in (0, 1)$. Clearly,~$\phi$ is a weak solution to problem~\eqref{projeq}--\eqref{ortcond} satisfying estimate~\eqref{linearLinftyest}. Let now~$(\hat{\phi}, \hat{d}_0, \hat{d}_1)$ be another bounded solution. Then, the difference~$\psi := \phi - \hat{\phi}$ solves
$$
\begin{dcases}
L_{\mu, 0} \, \psi = \sum_{i = 0}^1 e_i \chi Z_{i, \mu, 0} & \quad \mbox{in } \R, \\
\int_{\R} \psi \chi Z_{i, \mu, 0} \, dx = 0 & \quad \mbox{for } i = 0, 1,
\end{dcases}
$$
with~$e_i := d_i - \hat{d}_i$. By testing the equation for~$\psi$ against the functions~$\tilde{z}_0$,~$\tilde{z}_1$ of Lemma~\ref{2funclem} and arguing as in the proof of Lemma~\ref{existinboundedlem}, one easily obtains that~$e_0 = e_1 = 0$. Hence,~$\psi$ is an entire solution to~$L_{\mu, 0} \, \psi = 0$, orthogonal to~$\chi Z_{0, \mu, 0}$ and~$\chi Z_{1, \mu, 0}$. The nondegeneracy of~$L_{\mu, 0}$ (Lemma~\ref{nondeglem}) gives that~$\psi = 0$ in~$\R$. The proof of the proposition is thus complete.
\end{proof}

\subsection{Nonlinear theory} \label{subsectNLT} 

\begin{proposition} \label{mainnonlinprop}
There exist two constants~$\varepsilon_0 \in (0, 1)$ and~$C_\sharp > 0$, depending only on~$\mu_0$,~$\sigma$,~$\bar{R}$, and~$\| \kappa \|_{L^\infty(\R)}$, such that, if~$\varepsilon \in (0, \varepsilon_0]$,~$\mu \in \left[ \frac{1}{\mu_0}, \mu_0 \right]$, and~$\xi \in \R$, then there exists a unique bounded solution~$\phi$ of
\begin{equation} \label{projnonlineq}
L_{\mu, \xi} \, \phi = - \E + \NN(\phi) + \sum_{i = 0}^1 d_i \chi_\xi Z_{i, \mu, \xi} \quad \mbox{in } \R,
\end{equation}
for some uniquely determined~$d_0, d_1 \in \R$, which satisfies the orthogonality conditions~\eqref{ortcond} and the bound
\begin{equation} \label{nonlinLinftybound}
\| \phi \|_{L^\infty(\R)} \le C_\sharp \, \varepsilon.
\end{equation}
\end{proposition}
\begin{proof}
Given a function~$g \in L^\infty_{\sigma, \xi}(\R)$, let~$L_{\mu, \xi}^{-1}(g)$ be the unique bounded weak solution of problem~\eqref{projeq}--\eqref{ortcond}, for some~$d_0, d_1 \in \R$, as given by Proposition~\ref{mainlinearprop}. Thanks to estimate~\eqref{linearLinftyest}, we have that~$L_{\mu, \xi}^{-1}: L^\infty_{\sigma, \xi}(\R) \to L^\infty(\R)$ is a bounded linear operator and it holds
\begin{equation} \label{L-1bound}
\| L_{\mu, \xi}^{-1}(g) \|_{L^\infty(\R)} \le C_0 \| g \|_{L^\infty_{\sigma, \xi}(\R)},
\end{equation}
for some constant~$C_0 > 0$ depending only on~$\mu_0$,~$\sigma$, and~$\bar{R}$.

Notice that~$\E \in L^\infty_{\sigma, \xi}(\R)$, thanks to~\eqref{Ebound} of Lemma~\ref{ENboundslem}. Let then~$\tau > 0$ and consider the closed ball
$$
B_\tau := \Big\{ {\phi \in L^\infty(\R) : \| \phi \|_{L^\infty(\R)} \le \tau \varepsilon} \Big\}.
$$
If~$\phi \in B_\tau$ with~$\tau \varepsilon \le 1$, then also~$\NN(\phi) \in L^\infty_{\sigma, \xi}(\R)$, by~\eqref{Nbound}. By virtue of this, we may recast the problem of finding a bounded weak solution to~\eqref{projnonlineq} which satisfies the orthogonality conditions~\eqref{ortcond} as a fixed point problem for the map~$T: B_\tau \to L^\infty(\R)$ defined by~$T(\phi) := L_{\mu, \xi}^{-1} \left( - \E + \NN(\phi) \right)$ for~$\phi \in B_\tau$.

In order to prove the existence of a fixed point in~$B_\tau$, we employ the contraction lemma. To this aim, we shall show that
\begin{equation} \label{TmapsBrhoinitself}
\| T(\phi) \|_{L^\infty(\R)} \le \tau \varepsilon \quad \mbox{for all } \phi \in B_\tau
\end{equation}
and
\begin{equation} \label{Tcontracts}
\| T(\phi) - T(\psi) \|_{L^\infty(\R)} \le \frac{1}{2} \| \phi - \psi \|_{L^\infty(\R)} \quad \mbox{for all } \phi, \psi \in B_\tau,
\end{equation}
provided~$\tau$ is carefully chosen and~$\varepsilon$ is small enough. To verify~\eqref{TmapsBrhoinitself}, we recall~\eqref{Ebound},~\eqref{Nbound},~\eqref{L-1bound}, and deduce that, for~$\phi \in B_\tau$,
$$
\| T(\phi) \|_{L^\infty(\R)} \le C_0 \left\| - \E + \NN(\phi) \right\|_{L^\infty_{\sigma, \xi}(\R)} \le C_0 \left( C_1 + C_2 \tau (\tau + 1) \varepsilon \right) \varepsilon,
$$
with~$C_1$ and~$C_2$ given by Lemma~\ref{ENboundslem}. Consequently,~\eqref{TmapsBrhoinitself} holds true, provided we take~$\tau := 2 C_0 C_1$ and~$\varepsilon$ sufficiently small, in dependence of~$\mu_0$,~$\sigma$,~$\bar{R}$, and~$\| \kappa \|_{L^\infty(\R)}$ only. We now address~\eqref{Tcontracts}. Given~$\phi, \psi \in B_\tau$, using estimate~\eqref{NLipbound} we find that
$$
\| T(\phi) - T(\psi) \|_{L^\infty(\R)} \le C_0 \left\| \NN(\phi) - \NN(\psi) \right\|_{L^\infty_{\sigma, \xi}(\R)} \le C_0  C_2 \left( 2 \tau + 1 \right) \varepsilon \| \phi - \psi \|_{L^\infty(\R)},
$$
which gives~\eqref{Tcontracts} by taking~$\varepsilon$ small enough. In light of~\eqref{TmapsBrhoinitself} and~\eqref{Tcontracts}, the map~$T$ is a contraction of~$B_\tau$ into itself and thus it has a unique fixed point within~$B_\tau$. This concludes the proof.
\end{proof}

Proposition~\ref{mainnonlinprop} gives the existence of a solution~$\phi = \phi_{(\xi, \mu)}$ to the projected equation~\eqref{projnonlineq} in correspondence to every~$(\xi, \mu) \in \R \times \left[ \frac{1}{\mu_0}, \mu_0 \right]$. For later purposes, we establish here the continuity of~$\phi_{(\xi, \mu)}$ in~$(\xi, \mu)$ with respect to pointwise convergence.

\begin{lemma} \label{continuitylem}
Let~$\big\{ {(\xi_n, \mu_n)} \big\}_n \subset \R \times \left[ \frac{1}{\mu_0}, \mu_0 \right]$ be a sequence converging to some limit point~$(\xi_\infty, \mu_\infty)$. Then,~$\phi_{(\xi_n, \mu_n)} \rightarrow \phi_{(\xi_\infty, \mu_\infty)}$ locally uniformly in~$\R$. 
\end{lemma}
\begin{proof}
Write simply~$\phi_n := \phi_{(\xi_n, \mu_n)}$ and~$\phi_\infty :=\phi_{(\xi_\infty, \mu_\infty)}$. By estimate~\eqref{nonlinLinftybound} of Proposition~\ref{mainnonlinprop}, we know that~$\| \phi_\infty \|_{L^\infty(\R)} \le C_\sharp \, \varepsilon$ and~$\| \phi_n \|_{L^\infty(\R)} \le C_\sharp \, \varepsilon$ for every~$n \in \N$. Note that each~$\phi_n$ solves the equation
$$
(-\Delta)^{\frac{1}{2}} \phi_n = e^{\U_{\mu_n, \xi_n}} \phi_n - \E_{\mu_n, \xi_n} + \NN_{\mu_n, \xi_n}[\phi_n] + \sum_{i = 0}^1 d_{i, n} \chi_{\xi_n} Z_{i, \mu_n, \xi_n} \quad \mbox{in } \R,
$$
for some~$d_{0, n}, d_{1, n} \in \R$. Since~$(\xi_n, \mu_n)$ converges to~$(\xi_\infty, \mu_\infty)$, it is not hard to show that the right-hand side of the above equation has~$L^\infty(\R)$ norm bounded uniformly in~$n$. By this and standard elliptic estimates for~$(-\Delta)^{\frac{1}{2}}$, we infer that~$\| \phi_n \|_{C^\alpha(\R)}$ is also uniformly bounded in~$n$, for any~$\alpha \in (0, 1)$. By Ascoli-Arzel\`a theorem,~$\phi_n$ thus converges to some~$\hat{\phi}_\infty \in C^\alpha(\R)$ locally uniformly in~$\R$. Clearly,~$\hat{\phi}_\infty$ satisfies
$$
\begin{dcases}
L_{\mu_\infty, \xi_\infty} \, \hat{\phi}_\infty = - \E_{\mu_\infty, \xi_\infty} + \NN_{\mu_\infty, \xi_\infty}[\hat{\phi}_\infty] + \sum_{i = 0}^1 \hat{d}_{i, \infty} \chi_{\xi_\infty} Z_{i, \mu_\infty, \xi_\infty} & \quad \mbox{in } \R, \\
\int_{\R} \hat{\phi}_\infty \chi_{\xi_\infty} Z_{i, \mu_\infty, \xi_\infty} \, dx = 0 & \quad \mbox{for } i = 0, 1,
\end{dcases}
$$
for some~$\hat{d}_{0, \infty}, \hat{d}_{1, \infty} \in \R$. However, since we also have that~$\| \hat{\phi}_\infty \|_{L^\infty(\R)} \le C_\sharp \, \varepsilon$, the uniqueness granted by Proposition~\ref{mainnonlinprop} enables us to conclude that~$\hat{\phi}_\infty = \phi_\infty$ in~$\R$. This concludes the proof.
\end{proof}

\subsection{Finite dimensional reduction} \label{subsectFDR}

Thanks to Proposition~\ref{mainnonlinprop}, the problem of finding a solution~$\phi$ to~\eqref{maineqforphi} with~$\| \phi \|_{L^\infty(\R)} = O(\varepsilon)$ is thus reduced to finding~$\mu > 0$ and~$\xi \in \R$ for which the constants~$d_0 = d_0(\xi, \mu)$ and~$d_1 = d_1(\xi, \mu)$ appearing in~\eqref{projnonlineq} are both equal to zero.

We have the following result, which connects the vanishing of the coefficients~$d_i$'s to an equation for the gradient of the function~$\Gamma$ defined in~\eqref{Gamma0def}.

\begin{lemma} \label{dvanishlem}
Let~$(\xi, \mu) \in \R \times (0, +\infty)$. Then,~$d_0(\xi, \mu) = d_1(\xi, \mu) = 0$ if and only if
$$
\nabla \Gamma(\xi, \mu) = \frac{1}{2 \pi \varepsilon} \left( \int_{\R} \NN(\phi) Z_{1, \mu, \xi} \, dx, \int_{\R} \NN(\phi) Z_{0, \mu, \xi} \, dx \right).
$$
\end{lemma}
\begin{proof}
In order to find an expression for~$d_0$ and~$d_1$, we argue as in the proof of Lemma~\ref{existinboundedlem} and test equation~\eqref{projnonlineq} against the translations~$\tilde{z}_i(\cdot - \xi)$ of the functions constructed in Lemma~\ref{2funclem}, for~$r$ large. We obtain
$$
d_i \int_{\R} \chi_\xi Z_{i, \mu, \xi}^2 \, dx = \int_{\R} \Big( {\phi L_{\mu, \xi} \tilde{z}_{i, \xi} + \E \tilde{z}_{i, \xi} - \NN(\phi) \tilde{z}_{i, \xi}} \Big) \, dx,
$$
for~$i = 0, 1$ and provided~$r$ is sufficiently large. Recalling Lemmas~\ref{ENboundslem} and~\ref{2funclem}, by means of Lebesgue's dominated convergence theorem it is not hard to see that
$$
\int_{\R} \phi L_{\mu, \xi} \tilde{z}_{i, \xi} \, dx \rightarrow 0 \quad \mbox{and} \quad \int_{\R} \Big( {\E \tilde{z}_{i, \xi} - \NN(\phi) \tilde{z}_{i, \xi}} \Big) \, dx \rightarrow \int_{\R} \Big( {\E Z_{i, \mu, \xi} - \NN(\phi) Z_{i, \mu, \xi}} \Big) \, dx,
$$
as~$r \rightarrow +\infty$. Therefore,~$d_i = 0$ if and only if
$$
\int_{\R} \E Z_{i, \mu, \xi} \, dx = \int_{\R} \NN(\phi) Z_{i, \mu, \xi} \, dx.
$$

We now show that the term on the right is, up to a constant, a partial derivative of the function~$\Gamma$. We change variables and write
$$
\Gamma(\xi, \mu) = \frac{\mu}{\pi} \int_\R \frac{\kappa(x)}{\mu^2 + (x - \xi)^2} \, dx.
$$
Consequently, recalling~\eqref{Ealternateform} and~\eqref{Zdefs} we have that
\begin{equation} \label{eql29}
\begin{aligned}
\nabla \Gamma(\xi, \mu) & = \frac{1}{\pi} \left( \int_\R \frac{2 \mu (x - \xi)}{\big( {\mu^2 + (x - \xi)^2} \big)^2} \, \kappa(x) \, dx, \int_\R \frac{(x - \xi)^2 - \mu^2}{\big( {\mu^2 + (x - \xi)^2} \big)^2} \, \kappa(x) \, dx \right) \\
& = \frac{1}{2 \pi \varepsilon} \left( \int_\R \E Z_{1, \mu, \xi} \, dx, \int_\R \E Z_{0, \mu, \xi} \, dx \right).
\end{aligned}
\end{equation}
From this, the claim follows.
\end{proof}

Thanks to the previous lemma, we can now finish the proof of Theorem~\ref{mainthm}.

\begin{proof}[Completion of the proof of Theorem~\ref{mainthm}]
Let~$(\xi_\star, \mu_\star) \in \R \times (0, +\infty)$ be a non-degenerate critical point of~$\Gamma$. Then,
$$
\nabla \Gamma(\xi, \mu) = \left( \xi - \xi_\star, \mu - \mu_\star \right) M + \mathcal{R}_1(\xi, \mu), \quad \mbox{with } |\mathcal{R}_1(\xi, \mu)| \le C \Big( {(\xi - \xi_\star)^2 + (\mu - \mu_\star)^2} \Big),
$$
for all~$(\xi, \mu) \in \overline{B}_\delta := \overline{B_\delta \big( {(\xi_\star, \mu_\star)} \big)}$, for some constant~$C > 0$, some radius~$\delta \in \left( 0, \frac{\mu_\star}{2} \right]$, and some invertible matrix~$M \in \Mat_2(\R)$. Set
$$
\mathcal{R}_2(\xi, \mu) := \frac{1}{2 \pi \varepsilon} \left( \int_{\R} \NN(\phi) Z_{1, \mu, \xi} \, dx, \int_{\R} \NN(\phi) Z_{0, \mu, \xi} \, dx \right).
$$
In light of Proposition~\ref{mainnonlinprop} and Lemma~\ref{dvanishlem}, the claim of Theorem~\ref{mainthm} boils down to showing that the map
$$
\Xi(\xi, \mu) := (\xi_\star, \mu_\star) + \big( {\mathcal{R}_2(\xi, \mu) - \mathcal{R}_1(\xi, \mu)} \big) M^{-1}
$$
admits a fixed point~$(\xi_\varepsilon, \mu_\varepsilon) \in \overline{B}_\delta$ for every small~$\varepsilon$. But this is a consequence of Brouwer's fixed point theorem, since~$\Xi$ is continuous (thanks to Lemma~\ref{continuitylem} and dominated convergence) and maps~$\overline{B}_\delta$ into itself, provided~$\varepsilon$ and~$\delta$ are suitably small. Indeed, using estimates~\eqref{Nbound} and~\eqref{nonlinLinftybound}, we compute
\begin{align*}
\left| \Xi(\xi, \mu) - (\xi_\star, \mu_\star) \right| & \le C \Big( {|\mathcal{R}_1(\xi, \mu)| + |\mathcal{R}_2(\xi, \mu)|} \Big) \\
& \le C \left( (\xi - \xi_\star)^2 + (\mu - \mu_\star)^2 + \frac{\| \NN(\phi) \|_{L^\infty_{\sigma, \xi}(\R)}}{\varepsilon} \int_\R \frac{dx}{(1 + |x - \xi|)^{1 + \sigma}} \right) \\
& \le C ( \delta^2 + \varepsilon) \le \delta,
\end{align*}
choosing, for instance,~$\delta = \sqrt{\varepsilon}$ and taking~$\varepsilon$ small enough. The proof is thus concluded.
\end{proof}

We conclude the section with the

\begin{proof}[Proof of Corollary \ref{corCMepsilon}]
The result immediately follows from Theorem \ref{mainthm} arguing as in  \cite[Section 4.1]{AL22} (see also \cite[page 5]{AL22}).
\end{proof}

\section{Proof of Theorem \ref{optimality}} \label{sect3}
 


\begin{proof}[Proof of Theorem \ref{optimality}]
Let~$\sigma \in (0,1)$ be fixed. We also fix~$\epsilon_0 \in (0,1)$ small enough to ensure that~$C_0 \epsilon^{\alpha} \leq \sigma/2$,~$|\xi_{\epsilon}-\xi| \leq 1$, and~$|\mu_{\epsilon} - \mu| \leq \mu/10$ for all~$\epsilon \in (0, \epsilon_0)$. Let now~$\epsilon \in (0,\epsilon_0)$ be fixed. Arguing as at the beginning of Section~\ref{thm1.1proofsec}, it is immediate to see that~$\phi_{\epsilon}$ is a solution to~\eqref{maineqforphi} with~$(\xi,\mu) = (\xi_{\epsilon},\mu_{\epsilon})$. Testing~\eqref{maineqforphi} against the translations~$\tilde{z}_i(\cdot - \xi_{\epsilon})$ of the functions constructed in Lemma~\ref{2funclem}, we obtain 
$$
 \int_{\R}  {\phi_{\epsilon} L_{\mu_{\epsilon}, \xi_{\epsilon}} \tilde{z}_{i, \xi_{\epsilon}} \, dx = \int_{\R} \Big( {- \E \tilde{z}_{i, \xi_{\epsilon}} + \NN(\phi_{\epsilon}) \tilde{z}_{i, \xi_{\epsilon}}}} \Big) \, dx,
$$
for~$i = 0, 1$ and provided~$r$ is sufficiently large. Arguing as in the proof of Lemma~\ref{dvanishlem}, we send~$r$ to infinity and get that
$$
\int_{\R} \E Z_{i, \mu_{\epsilon}, \xi_{\epsilon}} \, dx = \int_{\R} \NN(\phi_\epsilon) Z_{i, \mu_{\epsilon}, \xi_{\epsilon}} \, dx, \quad \textup{ for } i = 0,1.
$$
Also, by~\eqref{eql29}, we know that
\begin{align*}
\nabla \Gamma(\xi_{\epsilon}, \mu_{\epsilon}) = \frac{1}{2 \pi \varepsilon} \left( \int_\R \E Z_{1, \mu_{\epsilon}, \xi_{\epsilon}} \, dx, \int_\R \E Z_{0, \mu_{\epsilon}, \xi_{\epsilon}} \, dx \right).
\end{align*}
Thus, for all~$\epsilon \in (0,\epsilon_0)$ we have that
\begin{equation} \label{nablagammaoptimality}
\nabla \Gamma(\xi_{\epsilon}, \mu_{\epsilon}) = \frac{1}{2 \pi \varepsilon} \left( \int_{\R} \NN(\phi_{\epsilon}) Z_{1, \mu_{\epsilon}, \xi_{\epsilon}} \, dx, \int_{\R} \NN(\phi_\epsilon) Z_{0, \mu_{\epsilon}, \xi_{\epsilon}} \, dx \right).
\end{equation}
On the other hand, using~\eqref{estimateNNphi}, we easily get
\begin{equation} \label{nonlinearoptimality}
\begin{aligned}
& \left| \int_{\R} \NN(\phi_{\epsilon}) Z_{i,\mu_{\epsilon},\xi_{\epsilon}}\, dx \right| \\
& \quad \leq \int_{\R} \frac{ \mu_{\epsilon}}{\mu_{\epsilon}^2 + (x - \xi_{\epsilon})^2} \bigg( \frac{e^{|\phi_{\epsilon}(x)|}}{2} \, |\phi_{\epsilon}(x)|^2 + \varepsilon \| \kappa \|_{L^\infty(\R)} e^{|\phi_{\epsilon}(x)|} |\phi_{\epsilon}(x)|  \bigg) |Z_{i,\mu_{\epsilon},\xi_{\epsilon}}(x)| \, dx \\
& \quad \leq C \big( \epsilon^{2\alpha} + \|\kappa\|_{L^{\infty}(\R)} \epsilon^{1+\alpha} \big) \int_{\R}  \frac{ \mu_{\epsilon}}{\mu_{\epsilon}^2 + (x - \xi_{\epsilon})^2} (2+|x-\xi|)^{\sigma}  |Z_{i,\mu_{\epsilon},\xi_{\epsilon}}(x)|\, dx \\
& \quad \leq  C \big( \epsilon^{2\alpha} + \|\kappa\|_{L^{\infty}(\R)} \epsilon^{1+\alpha} \big) \int_{\R} \frac{(1+|y|)^{\sigma}}{1+y^2} \, |Z_{i,1,0}(y)|\, dy, \quad \textup{ for } i = 0,1.
\end{aligned}
\end{equation}
Note that the constants~$C > 0$ depend only on~$C_0$,~$\mu$,~$\xi$,~$\sigma$ and that may change from line to line. Taking into account the boundedness of~$Z_{i,1,0}$ (recall definition~\eqref{Zdefs}) and combining~\eqref{nablagammaoptimality}--\eqref{nonlinearoptimality} with the continuity of~$\nabla \Gamma$, the result immediately follows. 
\end{proof}

\section{The functional $\Gamma$. Proof of Theorem~\ref{main2} and Corollary~\ref{corMepsilon2}} \label{section4}

\noindent
In view of the previous sections, and in particular Theorem~\ref{mainthm}, the existence of solutions to \eqref{maineq} can be reduced to finding non-degenerate critical points~$(\xi,\mu)\in\R_{+}^2$ for $\Gamma$. Therefore, from now on we will focus on the study of the critical points of~$\Gamma$. We believe this problem to be of independent interest.

First of all, we study the asymptotic behaviour of~$\Gamma(\xi, \mu)$ as~$\mu$ tends to zero. The following result shows that the half-Laplacian of~$\kappa$ plays a major role. This is in contrast with the counterpart for the planar Liouville equation, where the standard Laplacian appears in the expansion, see~\cite[Lemma 10.2]{GP05}.


\begin{proposition}\label{asymptoticgamma}
Let~$\kappa \in C^{1, \alpha}(\R) \cap L^\infty(\R)$ for some~$\alpha > 0$. Then, for any fixed $\xi\in\R$, 
$$\Gamma(\xi,\mu)=\kappa(\xi)-\mu(-\Delta)^\frac12\kappa(\xi)+O\left(\mu^{1 + \alpha}\right), \quad \textup{ as } \mu \to 0^{+}.
$$
\end{proposition}

\begin{remark} \label{fraclapk}
As is well-known, the assumption~$\kappa \in C^{1, \alpha}(\R) \cap L^{\infty}(\R)$ ensures that the half-Laplacian of~$\kappa$ is well-defined at any point. Indeed, note that
\begin{equation*}
\begin{aligned}
\left|\int_\R \frac{2\kappa(x) - \kappa(x + y)-\kappa(x - y)}{y^2} \, dy\right| & \le 4 \|\kappa\|_{L^\infty(\R)} \int_{\R \setminus (- 1, 1)} \frac{dy}{y^2} + [\kappa']_{C^{0, \alpha}(\R)} \int_{(-1,1)} \frac{|y|^{1 + \alpha}}{y^2} \, dy \\
& \le 8 \|\kappa\|_{L^{\infty}(\R)} + \frac{2}{\alpha} \, [\kappa']_{C^{0, \alpha}(\R)}.
\end{aligned}
\end{equation*}
\end{remark}

\begin{proof}[Proof of Proposition \ref{asymptoticgamma}]
The proof is by direct calculation. We estimate the difference~$\Gamma(\xi,\mu)-\kappa(\xi)$. Since, for any~$\xi\in\R,\mu>0$, 
$$\int_\R e^{\mathcal U_{\mu,\xi}}=\int_\R\frac{2\mu}{\mu^2+(x-\xi)^2}\,dx=2\pi,$$
using a change of variable, we get that
\begin{align*}
\Gamma(\xi,\mu) - \kappa(\xi) & = \frac1{\pi}\int_\R \frac{\kappa(\xi + \mu y)}{1 + y^2} \, dy-\kappa(\xi)\frac1{\pi}\int_\R \frac{\mu}{\mu^2+(x-\xi)^2} \, dx \\
& = \frac1{\pi}\int_\R(\kappa(x)-\kappa(\xi))\frac{\mu}{\mu^2+(x-\xi)^2} \, dx \\
& = \frac1{\pi}\,\PV\int_\R(\kappa(x)-\kappa(\xi))\left(\frac{\mu}{(x-\xi)^2}-\frac{\mu^3}{(x-\xi)^2\left(\mu^2+(x-\xi)^2\right)}\right) \, dx\\
& = -\mu (-\Delta)^{\frac12} \kappa (\xi) + \frac1\pi \, \PV \int_{\R} \frac{\kappa(x)-\kappa(\xi)}{(x-\xi)^2} \frac{-\mu^3}{\mu^2+(x-\xi)^2} \, dx, \quad \textup{ for all } (\xi,\mu) \in \R_{+}^2.
\end{align*}
On the other hand, arguing as in Remark \ref{fraclapk}, we get that
\begin{align*}
&  \frac1\pi \, \PV \int_{\R} \frac{\kappa(x)-\kappa(\xi)}{(x-\xi)^2} \frac{-\mu^3}{\mu^2+(x-\xi)^2} \, dx = \frac{1}{2\pi} \int_{\R} \frac{2\kappa(\xi) - \kappa(\xi + y)-\kappa(\xi - y)}{y^2} \frac{\mu^3}{\mu^2+y^2} \, dy = O(\mu^{1 + \alpha}),
\end{align*}
as $\mu \to 0^{+}$. Hence, the claim readily follows. 
\end{proof}

The previous proposition allows to extend $\Gamma$ to some $\widetilde\Gamma$ defined also for non-positive values of $\mu$.

\begin{corollary}\label{cortildegamma}
Let~$\kappa \in C^{1, \alpha}(\R) \cap L^\infty(\R)$ for some~$\alpha > 0$. Then, the map~$\widetilde\Gamma:\R^2 \to\R$, defined by
\begin{equation}\label{tildegamma}
\widetilde\Gamma(\xi,\delta)=\left\{
\begin{array}{ll}\Gamma\left(\xi,\delta^2\right)&\mbox{if }\delta\ne0,\\\kappa(\xi)&\mbox{if }\delta=0,\end{array}\right.
\end{equation}
is of class~$C^1$ and, for any fixed~$\xi \in \R$, it follows that
$$
\widetilde\Gamma(\xi,\delta)=\kappa(\xi)-\delta^2(-\Delta)^\frac12\kappa(\xi)+O \! \left(|\delta|^{2 + 2 \alpha}\right), \quad \textup{ as } |\delta| \to 0.
$$
In particular, a point in the form~$(\xi,0)$ is critical for~$\widetilde\Gamma$ if and only if~$\xi$ is critical for~$\kappa$. Moreover:
\begin{itemize}
\item If $\xi$ is a local maximum for $\kappa$ and $(-\Delta)^\frac12\kappa(\xi)>0$, then $(\xi,0)$ is a local maximum for $\widetilde\Gamma$;
\item If $\xi$ is a local maximum for $\kappa$ and $(-\Delta)^\frac12\kappa(\xi)<0$, then $(\xi,0)$ is a saddle point for $\widetilde\Gamma$;
\item If $\xi$ is a local minimum for $\kappa$ and $(-\Delta)^\frac12\kappa(\xi)>0$, then $(\xi,0)$ is a saddle point for $\widetilde\Gamma$;
\item If $\xi$ is a local minimum for $\kappa$ and $(-\Delta)^\frac12\kappa(\xi)<0$, then $(\xi,0)$ is a local minimum for $\widetilde\Gamma$.
\end{itemize}
\end{corollary}

\begin{proof}
The result is an immediate consequence of the asymptotic expansion given in Proposition~\ref{asymptoticgamma} and the second derivative test. 
\end{proof}

In order to get some information on the critical points of~$\Gamma$, we will compute its Brouwer degree. To that end, we need some extra assumptions on~$\kappa$ to avoid critical points at infinity. We require $\kappa$ to lie in $C^2_{*,\beta}(\R)$ (recall definition~\eqref{c*}-\eqref{c*norm}) and to satisfy \eqref{kprimneg} and \eqref{kintpos}.

\begin{lemma}\label{critinfinity}
Assume that $\kappa \in C^2_{*,\beta}(\R)$ for some~$\beta \in (0,1)$ satisfies~\eqref{kprimneg} and~\eqref{kintpos}. Then, there exists a constant~$R_1 > 0$ such that~$\nabla\Gamma(\xi,\mu)\cdot(\xi,\mu)<0$ for all~$(\xi, \mu) \in \R_{+}^2$ with~$|\xi| + \mu \geq R_1$.
\end{lemma}
\begin{proof}
We follow the argument of~\cite[Lemma 3.3]{AAP99}. Differentiating under the integral sign in~\eqref{Gamma0def} and changing variables, we get
$$
\nabla \Gamma(\xi, \mu) \cdot (\xi, \mu) = \frac{1}{\pi} \int_{\R} \frac{\kappa'(\xi + \mu y)(\xi + \mu y)}{1 + y^2} \, dy = \frac{\mu}{\pi} \int_{\R} \frac{\kappa'(z) z}{\mu^2 + (z - \xi)^2} \, dz.
$$
Now, recalling~\eqref{kprimneg}, for~$R \ge 2 R_0$ and~$|(\xi, \mu)| \ge R$ we have that
\begin{align*}
\int_{\R} \frac{\kappa'(z) z}{\mu^2 + (z - \xi)^2} \, dz & = \int_{- R_0}^{R_0} \frac{(\kappa'(z) z)_+}{|(\xi, \mu) - (z, 0)|^2} \, dz - \int_{\R} \frac{(\kappa'(z) z)_-}{|(\xi, \mu) - (z, 0)|^2} \, dz \\
& \le \frac{1}{\left( |(\xi, \mu)| - R_0 \right)^2} \int_{- R_0}^{R_0} (\kappa'(z) z)_+ \, dz - \frac{1}{\left( |(\xi, \mu)| + R_0 \right)^2} \int_{\R} (\kappa'(z) z)_- \, dz \\
& = \frac{1}{\left( |(\xi, \mu)| - R_0 \right)^2} \left\{ \int_{\R} \kappa'(z) z \, dz + \left( 1 - \frac{\left( |(\xi, \mu)| - R_0 \right)^2}{\left( |(\xi, \mu)| + R_0 \right)^2} \right) \int_{\R} (\kappa'(z) z)_- \, dz \right\} \\
& \le \frac{1}{\left( |(\xi, \mu)| - R_0 \right)^2} \left\{ \int_{\R} \kappa'(z) z \, dz + \left( 1 - \frac{\left( R - R_0 \right)^2}{\left( R + R_0 \right)^2} \right) \int_{\R} (\kappa'(z) z)_- \, dz \right\}.
\end{align*}
Thanks to hypothesis~\eqref{kintpos}, the quantity inside curly brackets is negative provided~$R$ is large enough. The proof is thus complete.
\end{proof}

The following result is crucial to deduce existence of solutions to~\eqref{maineq} from conditions on~$\kappa$. To this aim, we need to introduce some additional assumptions concerning~$\kappa \in C_{*,\beta}^2(\R)$. On top of~\eqref{kprimneg} and~\eqref{kintpos}, we require~$\kappa$ to satisfy that:
\begin{equation} \label{kmorse}
\textup{If } \kappa'(x) = 0, \textup{ then } \kappa''(x) \neq 0 \textup{ and } (-\Delta)^{\frac12}\kappa(x) \neq 0.
\end{equation}

\begin{theorem}\label{thmdeg}
Assume that~$\kappa \in C^{2}_{*,\beta}(\R)$ for some~$\beta \in (0, 1)$ satisfies~\eqref{kprimneg},~\eqref{kintpos}, and~\eqref{kmorse}. Moreover, let~$M^+_\kappa,m^+_\kappa \in \N \cup \{ 0 \}$ be given by
\begin{align*}
M^+_\kappa&:=\#\left\{\mbox{local maximum points } x \mbox{ for }\kappa\mbox{ satisfying }(-\Delta)^\frac12\kappa(x)>0\right\},\\
m^+_\kappa&:=\#\left\{\mbox{local minimum points } x \mbox{ for }\kappa\mbox{ satisfying } (-\Delta)^\frac12\kappa(x)>0\right\}.
\end{align*}
Then, there exists a constant~$R > 0$ such that~$\Gamma$ has no critical points outside~$B_R \cap \R^2_+$ and
\begin{equation}\label{degformula}
\deg \big( {\nabla\Gamma, B_R \cap \R_+^2, 0} \big) = 1 - M^+_\kappa + m^+_\kappa.
\end{equation}
In particular, if~$M^+_\kappa-m^+_\kappa\ne1$, then~$\Gamma$ has at least one critical point.
\end{theorem}

\begin{remark} \label{remarkTopologicalDegree}
We refer to~\cite[Chapter 3]{AM07} for the standard notation concerning the Brouwer degree we are going to use. Let us just remind here that, given a continuous vector field~$V$, its index at an isolated zero~$p$ satisfies
$$
{\rm{ind}}(V, p) = \lim_{r \to 0} \deg \big( {V, B_r(p), 0} \big).
$$
Also, let us stress that the index of the gradient of a continuously differentiable function~$f: \Omega \subset \R^2 \to \R$ at a non-degenerate minimum or maximum point of~$f$ is equal to~$+1$, while at a non-degenerate saddle point it is equal to~$-1$.
\end{remark}


\begin{proof}[Proof of Theorem \ref{thmdeg}]
Recalling definition~\eqref{tildegamma}, we consider the maps~$X: \overline{\R^2_+} \to \R^2$ and~$\widetilde{X}: \R^2 \to \R^2$ given by
$$
X(\xi, \delta) := \nabla \Gamma(\xi, \delta^2) \quad \mbox{and} \quad \widetilde{X}(\xi, \delta) := \nabla \widetilde\Gamma(\xi, \delta) = \left( \partial_\xi \Gamma(\xi, \delta^2), 2 \delta \partial_{\mu} \Gamma(\xi, \delta^2) \right).
$$
Thanks to the regularity assumptions on~$\kappa$ and Corollary~\ref{cortildegamma}, both~$X$ and~$\widetilde{X}$ are continuous in their respective domains. Observe that~$(\xi, \delta) \in \R^2_+$ is a zero of~$X$ if and only if it is a zero of~$\widetilde{X}$. Moreover, from~\eqref{kprimneg} and Lemma~\ref{critinfinity} we infer that there exists~$R \ge 1$ such that
\begin{equation} \label{tildeXnonzero}
\widetilde{X}(\xi, \delta) \cdot (2 \xi, \delta) < 0 \quad \mbox{for all } (\xi, \delta) \in \R^2 \setminus B_R.
\end{equation}
Also, since~$\Gamma$ is harmonic, all its critical points are isolated. By these considerations and~\eqref{kmorse}, we easily deduce that~$X$ and~$\widetilde{X}$ have a finite number of zeroes in their domains. In particular, taking $R$ larger if necessary, all the zeroes of both vector fields in~$\R^2_+$ are contained in~$D_+ := \left\{ (\xi, \delta) \in B_R : \delta > \frac{1}{R} \right\}$ and all the zeroes of~$\widetilde{X}$ on~$\R \times \{ 0 \}$ are contained in~$D_0 := \left\{ (\xi, \delta) \in B_R : |\delta| < \frac{1}{R} \right\}$---also note that, by~\eqref{kmorse},~$X$ has no zeroes on~$\R \times \{ 0 \}$.

We claim that~$X$ and~$\widetilde{X}$ are homotopically equivalent in~$D_+$, in the sense of~\cite[Page~28]{AM07}. Indeed, let
$$
h \big( {t, (\xi, \delta)} \big) := \left( \partial_\xi \Gamma(\xi, \delta^2), (1 + (2 \delta - 1) t) \partial_{\mu} \Gamma(\xi, \delta^2) \right) \quad \mbox{for } t \in [0, 1] \mbox{ and } (\xi, \delta) \in \R^2_+.
$$
Clearly,~$h$ is a continuous function satisfying~$h(0, \cdot) = X$ and~$h(1, \cdot) = \widetilde{X}$. Moreover,~$h(t, \cdot)$ does not vanish on~$\partial D_+$ for all~$t \in [0, 1]$ since~$X \ne 0$ there. From the invariance under homotopy of the Brouwer degree, we conclude that
\begin{equation} \label{degX=degXtilde}
\deg \big( {X, D_+, 0} \big) = \deg \big( {\widetilde{X}, D_+, 0} \big)
\end{equation}

Thanks to~\eqref{tildeXnonzero}, it is immediate to verify that~$\widetilde{X}$ is homotopically equivalent to the map~$A(\xi,\delta) := (- 2\xi, - \delta)$ in~$B_R$. Thus, its Brouwer degree is
$$
\deg \big( {\widetilde{X}, B_R, 0} \big) = \deg \! \big( {A, B_R, 0} \big) = 1.
$$
By this,~\eqref{degX=degXtilde}, and considerations made earlier, setting~$D_- := \left\{ (\xi, \delta) \in B_R : \delta < - \frac{1}{R} \right\}$ we infer from the decomposition property of the Brouwer degree that
\begin{equation} \label{1=d1+d2}
\begin{aligned}
1 & = \deg \big( {\widetilde{X}, B_R, 0} \big) = \deg \big( {\widetilde{X}, D_+, 0} \big) + \deg \big( {\widetilde{X}, D_-, 0} \big) + \deg \big( {\widetilde{X}, D_0, 0} \big) \\
& = 2 \deg \big( {X, D_+, 0} \big) + \deg \big( {\widetilde{X}, D_0, 0} \big).
\end{aligned}
\end{equation}

Let now~$M_\kappa^+$ and~$m_\kappa^+$ be as in the statement of the theorem and define
\begin{align*}
M^-_\kappa & := \#\left\{\mbox{local maximum points } x \mbox{ for }\kappa\mbox{ satisfying }(-\Delta)^\frac12\kappa(x)<0\right\},\\
m^-_\kappa & := \#\left\{\mbox{local minimum points } x \mbox{ for }\kappa\mbox{ satisfying }(-\Delta)^\frac12\kappa(x)<0 \right\}.
\end{align*}
By putting together Corollary~\ref{cortildegamma}, assumption~\eqref{kmorse}, and Remark~\ref{remarkTopologicalDegree}, we obtain that
$$
\deg \big( {\widetilde{X}, D_0, 0} \big) = \sum_{\xi \in (- R, R) \, : \, \kappa'(\xi) = 0} \mbox{ind} \big( {\widetilde{X}, (\xi, 0)} \big) = M_\kappa^+ - M_\kappa^- + m_\kappa^- - m_\kappa^+.
$$
Furthermore, in view of~\eqref{kprimneg}, the number of local maxima of~$\kappa$ exceeds exactly by one the number of its local minima. Hence,~$M_\kappa^+ + M_\kappa^- = m_\kappa^- + m_\kappa^+ + 1$ and the previous identity becomes
$$
\deg \big( {\widetilde{X}, D_0, 0} \big) = 2 ( M_\kappa^+ - m_\kappa^+ ) - 1.
$$
Claim~\eqref{degformula} immediately follows by combining this with~\eqref{1=d1+d2}.
\end{proof}

We can deduce some more information on the critical points of $\Gamma$ due to the fact that~$\Gamma$ is a bi-dimensional harmonic function. In particular, we deduce an exact multiplicity result provided all critical points are non-degenerate. This is a new feature for our problem with respect to previous results~\cite{AAP99,GP05}.

\begin{proposition}\label{propexactnum}
If the assumptions of Theorem~\ref{thmdeg} are satisfied and all the critical points of~$\Gamma$ are non-degenerate, then~$\Gamma$ has exactly~$M^+_\kappa-m^+_\kappa-1$ critical points. 
\end{proposition}


\begin{proof}
Thanks to Theorem~\ref{thmdeg}, there exists~$R > 0$ for which
$$
C:= \Big\{ {(\xi,\mu) \in \overline{\R_+^2} : \nabla \Gamma(\xi,\mu) = 0} \Big\} = \Big\{ {(\xi,\mu) \in B_R \cap \R_+^2: \nabla \Gamma(\xi,\mu) = 0} \Big\}.
$$
As the critical points of~$\Gamma$ are isolated---$\Gamma$ being harmonic---, we have that
\begin{equation} \label{conclussionProp48}
\deg \big( {\nabla \Gamma,B_R \cap \R_+^2,0} \big) = \sum_{(\xi,\mu) \in C} {\rm{ind}} \big( {\nabla \Gamma, (\xi,\mu)} \big).
\end{equation}
Now, since~$\Gamma$ is harmonic and all its critical points are non-degenerate, it follows that they can only be of saddle type, hence their index is always~$-1$. Thus, the result follows combining~\eqref{degformula} and~\eqref{conclussionProp48}.
\end{proof}

\begin{corollary} \label{corexactnum}
If the assumptions of Theorem \ref{thmdeg} are satisfied and all the critical points of $\Gamma$ are non-degenerate, then \eqref{maineq} has at least $M^+_\kappa-m^+_\kappa-1$ solutions of the form \eqref{u=U+phi}.
\end{corollary}

\begin{proof}
The result follows from the combination of Theorem \ref{mainthm} with Proposition \ref{propexactnum}.
\end{proof}

\begin{remark}\label{max-min}
If~$\kappa$ has only~$N+1$ global maximum points and~$N$ global minimum points then~$M^+_\kappa={N+1}$ and~$m^-_\kappa=0$ because~$(-\Delta)^\frac12 \kappa(x)>0$ if~$x$ is global maximum point and~$(-\Delta)^\frac12 \kappa(x)<0$ if~$x$ is a global minimum point. If, in addition, all critical points of~$\Gamma$ are non-degenerate then their number is exactly~$N$.
\end{remark}

Thanks to this observation, we are now almost in position to establish the validity of Theorem~\ref{main2}. In order to do this, we need the following density result.

\begin{lemma} \label{denslem}
Let~$\beta \in (0, 1)$. The following statements hold true.
\begin{enumerate}[label=$(\alph*)$,leftmargin=*]
\item \label{a} The set
$$
\mathscr{A} := \Big\{ {\kappa \in C^2_{*, \beta}(\R) : \mbox{its harmonic extension~$\Gamma$ has no degenerate critical points in~$\R^2_+$}} \Big\}
$$
is dense in~$C^2_{*, \beta}(\R)$.
\item \label{b} Considering
$$
\mathscr{X} := \Big\{ {\kappa \in C^2_{*, \beta}(\R) : \mbox{$\kappa$~satisfies~\eqref{kprimneg}}} \Big\},
$$
its subset~$\mathscr{B} := \mathscr{A} \cap \mathscr{X}$ is dense in~$\mathscr{X}$ (with respect to the induced metric).
\end{enumerate}
\end{lemma}
\begin{proof}
Let~$\kappa \in C^2_{*, \beta}(\R)$ and~$\Gamma \in C^0 \big( {\overline{\R^2_+}} \big) \cap L^\infty (\R^2_+) \cap C^\infty(\R^2_+)$ be its harmonic extension to~$\R^2_+$. Let~$\D := \big\{ {(x, y) \in \R^2 : x^2 + y^2 < 1} \big\}$ and~$\Phi: \overline{\D} \setminus \{ e_2 \} \to \overline{\R^2_+}$ be the conformal mapping
$$
\Phi(x, y) := \left( \frac{2x}{x^2 + (y- 1)^2}, \frac{1 - x^2 - y^2}{x^2 + (y - 1)^2} \right) \quad \mbox{for } (x, y) \in \overline{\D} \setminus \{ e_2 \},
$$
where~$e_2 = (0, 1)$ denotes the second element of the standard basis of~$\R^2$. Let~$\gamma := \Gamma \circ \Phi$. We have that~$\gamma$ lies in~$C^0(\overline{\D} \setminus \{ e_2 \}) \cap C^\infty(\D)$ and is harmonic in~$\D$.

Let~$V := \nabla \gamma$ and consider the set of its singular points
$$
\mathscr{S} := \Big\{ {p \in \D : \det JV(p) = 0} \Big\}.
$$
Observe that,~$\gamma$ being harmonic,~$\mathscr{S}$ coincides with the set of points at which the Jacobian matrix of~$V$ has rank equal to zero. In symbols,
$$
\mathscr{S} = \Big\{ {p \in \D : \mbox{rank } JV(p) = 0} \Big\}
$$
As~$V \in C^\infty(\D; \R^2)$, a suitable version of the Federer-Morse-Sard theorem (see,~e.g.,~\cite{M01}) yields that
$$
\mathcal{H}^1 \big( {V(\mathscr{S})} \big) = 0,
$$
where~$\mathcal{H}^1$ denotes the~$1$-dimensional Hausdorff measure in~$\R^2$. As a result, there exists an infinitesimal sequence~$\{ \varepsilon_j \} \subset (0, 1)$ such that
$$
\varepsilon_j e_2 \in \R^2 \setminus V(\mathscr{S}) \quad \mbox{for all } j \in \N.
$$
Hence,
$$
\det JV(p) \ne 0 \quad \mbox{for all } p \in D \mbox{ s.t.~} V(p) = \varepsilon_j e_2 \mbox{ and all } j \in \N.
$$

Set now
$$
\gamma_j(p) := \gamma(p) - \varepsilon_j (p - e_2) \cdot e_2  \quad \mbox{for all } p \in \overline{\D} \setminus \{ e_2 \}.
$$
Clearly,~$\gamma_j$ is harmonic in~$\D$. Moreover,
$$
\nabla \gamma_j = \nabla \gamma - \varepsilon_j e_2 = V - \varepsilon_j e_2 \qquad \mbox{and} \qquad
D^2 \gamma_j = D^2 \gamma = JV,
$$
so that~$\det D^2 \gamma_j \ne 0$ at all critical points of~$\gamma_j$ in~$\D$. That is,~$\gamma_j$ has no degenerate critical points in~$\D$.

Let now~$\Gamma_j := \gamma_j \circ \Psi$, with~$\Psi(\xi, \mu) := \left( \frac{2 \xi}{\xi^2 + (\mu + 1)^2}, \frac{\xi^2 + \mu^2 - 1}{\xi^2 + (\mu + 1)^2} \right)$ being the inverse function of~$\Phi$. We know that~$\Gamma_j \in C^0 \big( {\overline{\R^2_+}} \big) \cap L^\infty (\R^2_+) \cap C^\infty(\R^2_+)$ is harmonic   and has no degenerate critical points in~$\R^2_+$. In particular, it is the harmonic extension of its trace~$\kappa_j(\xi) := \Gamma_j(\xi, 0)$. We compute
$$
\kappa_j(\xi) = (\Gamma \circ \Phi \circ \Psi)(\xi, 0) - \varepsilon_j \left( \left( \frac{2 \xi}{1 + \xi^2}, \frac{\xi^2 - 1}{1 + \xi^2} \right) - (0, 1) \right) \cdot (0, 1) = \kappa(\xi) + \frac{2 \varepsilon_j}{1 + \xi^2}.
$$
From this we immediately see that~$\kappa_j \rightarrow \kappa$ in~$C^2_{*, \beta}(\R)$ as~$j \rightarrow +\infty$ and point~\ref{a} is verified. Furthermore, if~$\kappa$ satisfies~\eqref{kprimneg}, then there exists~$R_0 > 0$ such that~$\xi \kappa'(\xi) < 0$ for all~$\xi \in \R \setminus [- R_0, R_0]$. But then,
$$
\xi \kappa_j'(\xi) = \xi \kappa'(\xi) - \frac{4 \varepsilon_j \xi^2}{(1 + \xi^2)^2} < 0 \quad \mbox{for all } \xi \in \R \setminus [- R_0, R_0],
$$
that is,~$\kappa_j$ satisfies~\eqref{kprimneg} for all~$j \in \N$. This establishes the validity of point~\ref{b}.
\end{proof}

We refer to Appendix~\ref{tranapp} for a different proof of Lemma~\ref{denslem} \ref{a}. This alternative approach relies on the generalization of an abstract transversality result due to~Saut~\&~Temam~\cite{ST79} and is closer in spirit to arguments used in, say,~\cite{MP13,BHJY22} to establish genericity results. Having at hand Lemma~\ref{denslem}, we are now ready to present the

\begin{proof}[Proof of Theorem~\ref{main2}]
First of all, we point out that~$(-\Delta)^\frac12\kappa_0>0$ at any global maximum point and~$(-\Delta)^\frac12\kappa_0<0$ at any global minimum point. Since these extrema are all non-degenerate and~$\kappa_0$ has no other stationary point, we see that it satisfies~\eqref{kmorse}.

Let~$\Gamma_0$ be the harmonic extension of~$\kappa_0$ to~$\R^2_+$. If~$\Gamma_0$ has no degenerate critical points in~$\R^2_+$, then we may directly apply Proposition~\ref{propexactnum} (also recall Remark~\ref{max-min}) and deduce from Theorem~\ref{mainthm} that equation~\eqref{maineq}, for~$\kappa = \kappa_0$ and~$\varepsilon$ sufficiently small, has at least~$N$ solutions of the form~\eqref{u=U+phi}.

Suppose on the other hand that~$\Gamma_0$ has degenerate critical points. Thanks to point~\ref{b} of Lemma~\ref{denslem}, there exists a sequence~$\{ \kappa_j \} \subset C^2_{*, \beta}(\R)$ converging to~$\kappa_0$ in~$C^2_{*, \beta}(\R)$ and such that~$\kappa_j$ satisfies~\eqref{kprimneg} and its harmonic extension~$\Gamma_j$ has no degenerate critical points in~$\R^2_+$, for each~$j \in \N$. From this we also infer that, if~$j$ is sufficiently large,~$\kappa_j$ also has exactly~$N + 1$ local maxima,~$N$ local minima (all non-degenerate) and no other critical point. Moreover, its half-Laplacian is strictly positive at the maxima and strictly negative at the minima. In particular,~$\kappa_j$ satisfies~\eqref{kmorse}. It is immediate to check that it satisfies~\eqref{kintpos} as well, provided~$j$ is large enough. The conclusion then follows by applying Proposition~\ref{propexactnum} in conjunction with Theorem~\ref{mainthm}.
\end{proof}

\begin{proof}[Proof of Corollary \ref{corMepsilon2}]
The result follows arguing as in the proof of Corollary~\ref{corCMepsilon} using now Theorem~\ref{main2} instead of Theorem~\ref{mainthm}.
\end{proof}

\section{Examples and remarks} \label{section5}

\noindent In this section we provide some examples of functions $\kappa$ for which problem \eqref{maineq} has or does not have solutions. We will also make some comments and remarks, mostly comparing the results from the present paper with previous ones in the literature.

\begin{example}
In order to produce bounded harmonic functions in~$\R^2_+$ having non-degenerate critical points it is useful to consider the standard conformal mapping
$$
\Psi: \R^2_+ \to \mathbb{D} := \big\{ {(x, y) \in \R^2 : x^2 + y^2 < 1} \big\}
$$
defined by setting
\begin{equation}\label{psi}
\Psi(\xi, \mu) := \left( \frac{2 \xi}{\xi^2 + (\mu + 1)^2}, \frac{\xi^2 + \mu^2 - 1}{\xi^2 + (\mu + 1)^2} \right) \quad \mbox{for } (\xi, \mu) \in \R^2_+.
\end{equation}
Then, a function~$\Gamma$ is harmonic in~$\R^2_+$ if and only if~$\Gamma \circ \Psi^{-1}$ is harmonic in~$\mathbb{D}$ and~$(\xi_\star, \mu_\star) \in \R^2_+$ is a non-degenerate critical point for~$\Gamma$ if and only if~$\Psi^{-1}(\xi_\star, \mu_\star)$ is a non-degenerate critical point for~$\Gamma \circ \Psi^{-1}$. Hence, to obtain bounded harmonic functions in the unbounded set~$\R^2_+$ it suffices to find harmonic functions in bounded sets containing~$\mathbb{D}$. 
A class of functions that can be easily handled is that of harmonic polynomials. Following are two examples of harmonic polynomials generating admissible~$\kappa$'s.

\begin{enumerate}[leftmargin=*,label=$(\alph*)$]
\item First, we consider the function~$\gamma^{(1)}(x, y) = x^2 - y^2 + 1$. Notice that~$\gamma^{(1)}$ is harmonic and has a unique critical point at the origin, which is non-degenerate. Consequently, the function
$$
\Gamma^{(1)}(\xi, \mu) := \gamma^{(1)}(\Psi(\xi, \mu)) = 4 \, \frac{2 \xi^2 + \mu (\xi^2 + (\mu + 1)^2)}{(\xi^2 + (\mu + 1)^2)^2}
$$
is harmonic in~$\R^2_+$ and has a unique non-degenerate critical point at~$(0, 1)$. Its trace on~$\partial \R^2_+$ is the function
$$
\kappa^{(1)}(\xi) = \Gamma^{(1)}(\xi, 0) = \frac{8 \xi^2}{(1 + \xi^2)^2}.
$$

\item Next, we construct a harmonic polynomial with (at least) two symmetric non-degenerate critical points. A straightforward computation shows that
$$
\gamma^{(2)}(x, y) = - 2 x^4 - 2 y^4 + 12 x^2 y^2 - 4 y^3 + 12 x^2 y + x^2 - y^2 - 3 y + 10
$$
is a harmonic polynomial, even with respect to~$x$, and having non-degenerate critical points at~$\left( - \frac{1}{2}, 0 \right)$ and~$\left( \frac{1}{2}, 0 \right)$. Pulling~$\gamma^{(2)}$ back via the mapping~$\Psi$ leads us to the admissible function
$$
\kappa^{(2)}(\xi) = 2 \, \frac{75 \xi^6 - 35 \xi^4 + 25 \xi^2 + 7}{(1 + \xi^2)^4},
$$
whose harmonic extension~$\Gamma^{(2)}$ has non-degenerate critical points at~$\left( -\frac{4}{5}, \frac{3}{5} \right)$ and~$\left( \frac{4}{5}, \frac{3}{5} \right)$.
\end{enumerate}

\begin{figure}[H]
\centering
\fbox{
$\ $ 
\includegraphics[width=0.45\textwidth,  trim=-20pt -30pt -20pt -30pt  ]{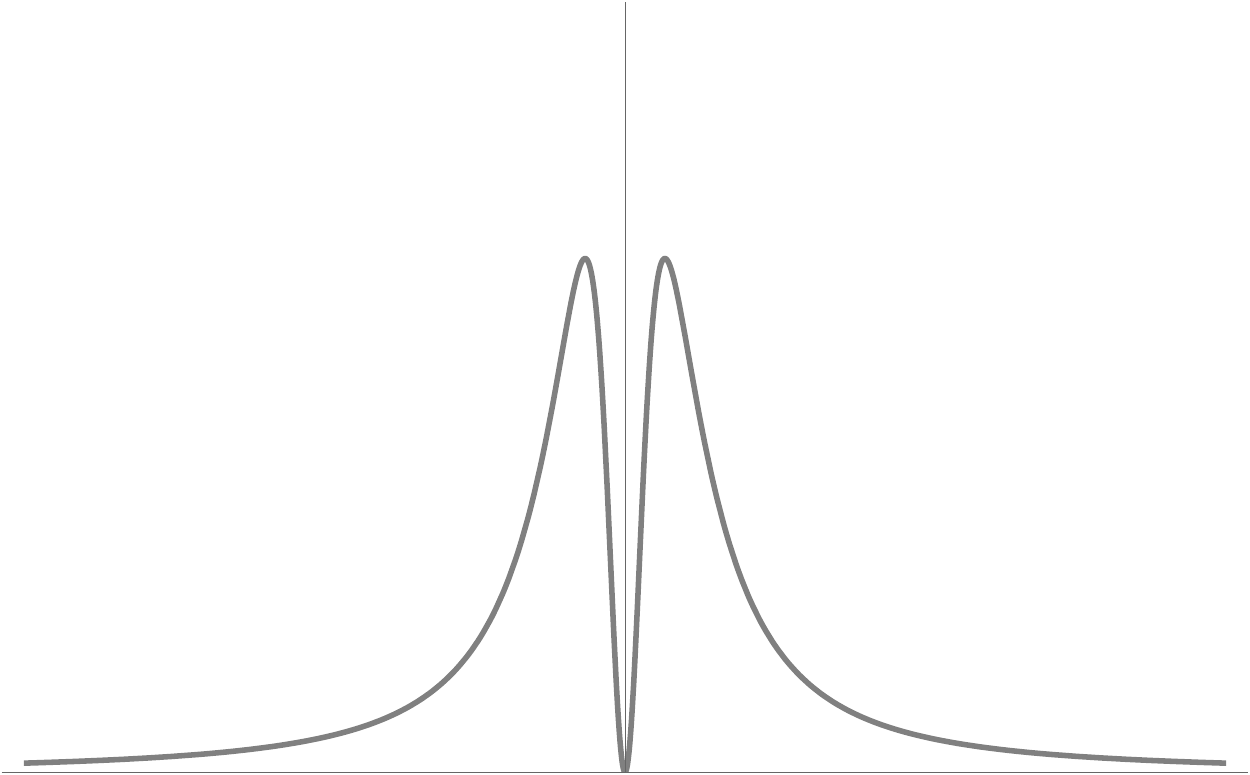} \includegraphics[width=0.45\textwidth, trim=-20pt -30pt -20pt -30pt]{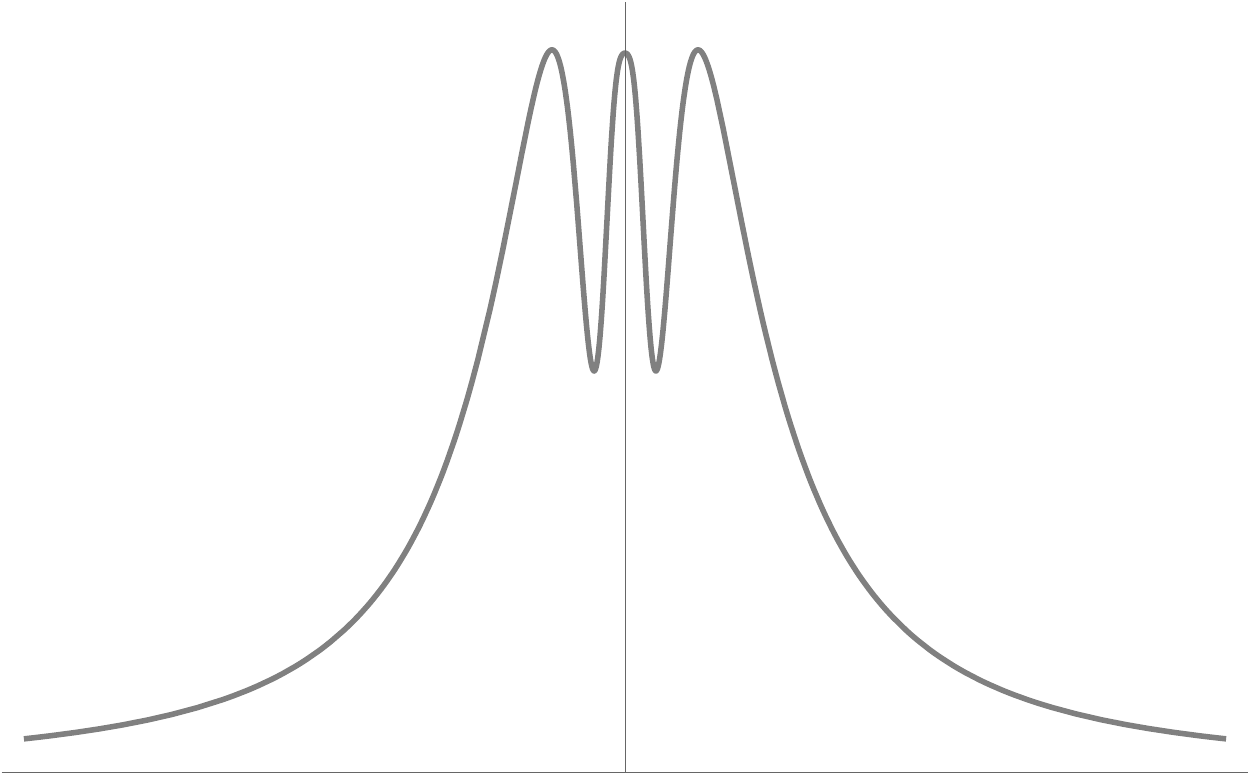} 
}
\caption{The graphs of the two admissible functions~$\kappa^{(1)}$ (on the left) and~$\kappa^{(2)}$ (on the right).}
\end{figure}
\end{example}
 
\begin{example}
On the negative direction, one can also find ``simple'' harmonic polynomials that have no critical points in $\mathbb{D}$ but have several critical points on $\partial \mathbb{D} = \mathbb{S}^1$. For instance, a straightforward computation shows that the function 
$$
\gamma^{(3)}(x,y) := \frac{1}{3} (x^3-3xy^2) - 2 x,
$$
is harmonic in $\R^2$, has no critical points in $\mathbb{D}$ but has $6$ critical points on $\mathbb{S}^1$.
\end{example}

In the next example we explore this direction in a more systematic way using some elementary complex analysis. 

\begin{example}
Let us consider, for~$N\in\N$ and~$a\ge0$, the smooth map~$k:\S^1 \to\R$ given by
$$
k \! \left( e^{it} \right) := k_{N,a} \! \left( e^{it} \right) := \frac{1-\cos\left((N+1)\left(t-\frac\pi2\right)\right)}{N+1}+a(1-\sin t),
$$
and let us define~$\kappa(\xi)=\kappa_{N,a}(\xi) := k(\Psi(\xi,0))$, where we use the same letter~$\Psi$ to indicate the continuous extension of the map in~\eqref{psi} to~$\overline{\R^2_+}$.

First of all, we claim that the function~$\kappa$ verifies~\eqref{kprimneg} and~\eqref{kintpos}. Indeed, a simple computation shows that
$$
\kappa'(\xi) = \frac{2}{1 + \xi^2} \left( \left. \frac{d}{dt} \, k(e^{i t}) \right|_{e^{i t} = \Psi(\xi, 0)} \right) = \frac{2}{1 + \xi^2} \left. \left\{ \sin \left( (N+1) \left(t-\frac\pi2\right) \right) - a \cos t \right\} \right|_{e^{i t} = \Psi(\xi, 0)},
$$
for all~$\xi \in \R$. From this we immediately conclude that~$\kappa'(\xi)>0$ for large~$\xi$ and~$\kappa'(\xi)<0$ for large~$-\xi$, that is,~\eqref{kprimneg} holds true. From the computation above one also gets a constant~$C > 0$ such that~$|\kappa'(\xi)|\le \frac C{1 + |\xi|^3}$ in~$\R$. This proves that the function~$\xi \mapsto \xi k'(\xi)$ belongs to~$L^1(\R)$. Arguing in the same way, one can easily see that~$|\kappa(\xi)| \le \frac C{1 + \xi^2}$. Hence,~$\kappa\in L^1(\R)$ and the fact that~$\kappa>0$ in~$\R$ immediately implies~\eqref{kintpos} after an integration by parts.

In order to count the critical points of~$\Gamma$ in~$\R_+^2$, it suffices to find the critical points of the harmonic extension~$\gamma$ of~$k$ in $\mathbb{D}$. Note that the harmonic extension of~$k$ to~$\mathbb{D}$ is explicitly given by
$$
\begin{aligned}
\gamma \! \left(re^{it}\right) & =\frac1{N+1}+a-\frac{r^{N+1}}{N+1}\cos\left((N+1)\left(t-\frac\pi2\right)\right)-ar\sin t \\
& =\Re\left(\frac1{N+1}+a-\frac{(-ire^{it})^{N+1}}{N+1}+aire^{it}\right).
\end{aligned}
$$
Therefore, critical points of~$\gamma$ correspond to critical points to the holomorphic map 
$$
z\mapsto\frac1{N+1}+a - \frac{(-iz)^{N+1}}{N+1} + aiz,
$$ 
namely, the critical points of~$\gamma$ are the~$N$-th roots of~$- i^N a$. We get one solution with multiplicity~$N$ if~$a=0$,~$N$ distinct solutions if~$0<a<1$ and, since we are interested in solutions in the open disk, no solutions if~$a\ge1$. Note that, when $ a > 0$, the critical points of $\gamma$ are non-degenerate.

In the case~$a=0$,~$\kappa$ has~$N+1$ maximum points at the same height and~$N$ minimum points at the same height; in particular, we get~$(-\Delta)^\frac12\kappa>0$ at each maximum and~$(-\Delta)^\frac12\kappa<0$ at each minimum, therefore~$M^+_\kappa=N+1$,~$m^+_\kappa=0$ and Theorem~\ref{thmdeg} and Corollary~\ref{corexactnum} can be applied. We just do not get actual multiplicity of solutions because they collapse at the same point.

If~$a>0$ is close to~$0$, by continuity one still gets~$N+1$ maxima with positive half-Laplacian and~$N$ minima with negative half-Laplacian, though not at the same height anymore. Therefore, Corollary~\ref{corexactnum} gives the existence of~$N$ solutions, consistently with the explicit analysis (which also show that solutions are now distinct). Computations also show that this picture holds true for any~$a\in(0,1)$. 

If~$a\ge1$, then we showed that there are no solutions, therefore Theorem~\ref{thmdeg} cannot be applied. Nonetheless,~$\kappa$ might still have a number of critical points, more than just a global maximum and a global minimum. Formula~\eqref{degformula} implies that the half-Laplacian on these points must have the ``wrong'' sign so that~$M^+_\kappa=m^+_\kappa-1$.

It is worthy to point out that~$\kappa$ could have an arbitrarily large number of critical points, although its harmonic extension has none. In particular, one can easily show that, if~$N + 2 \notin 4\N$ and~$a>1$ is close to~$1$, then~$\kappa$ has~$N+1$ maxima and~$N$ minima.


\begin{figure}[H]
\fbox{
\begin{tabular}{c}
\includegraphics[width=0.43\textwidth,  trim=-20pt 0pt -20pt -30pt ]{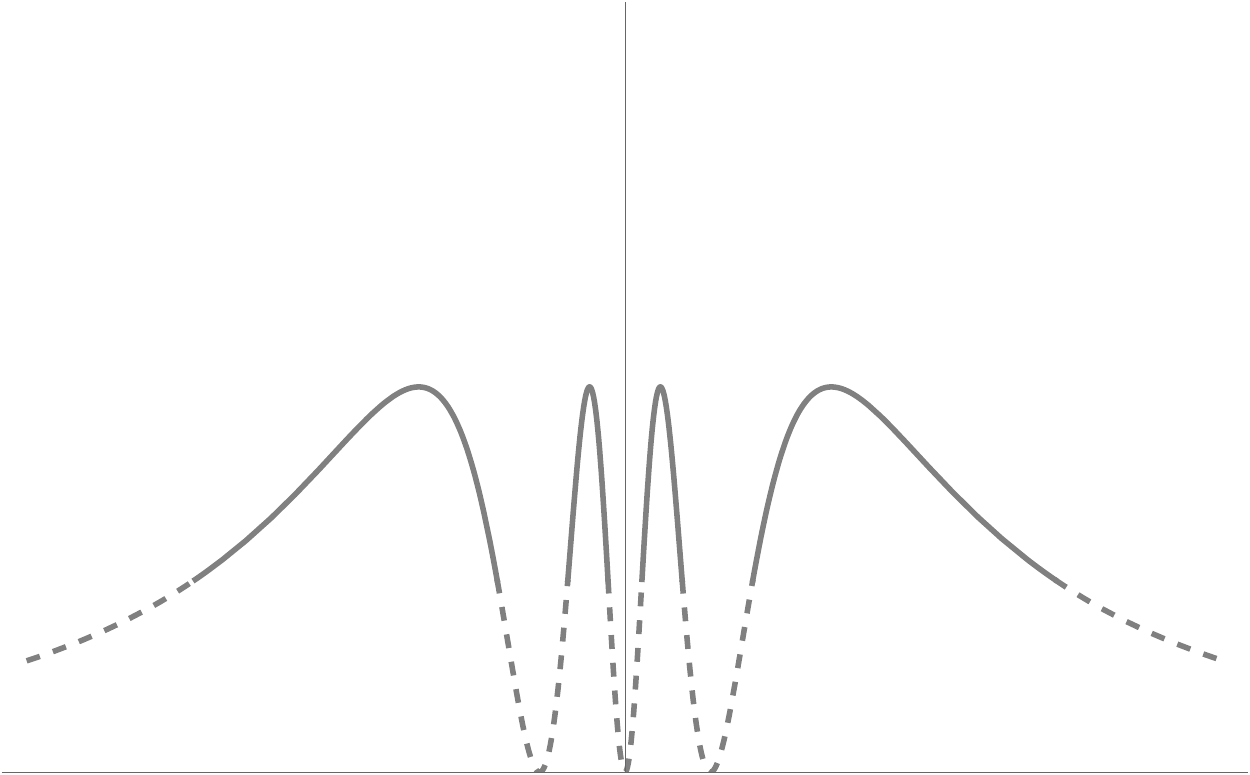}  \\
\includegraphics[width=0.43\textwidth,  trim=-20pt -30pt -20pt -20pt ]{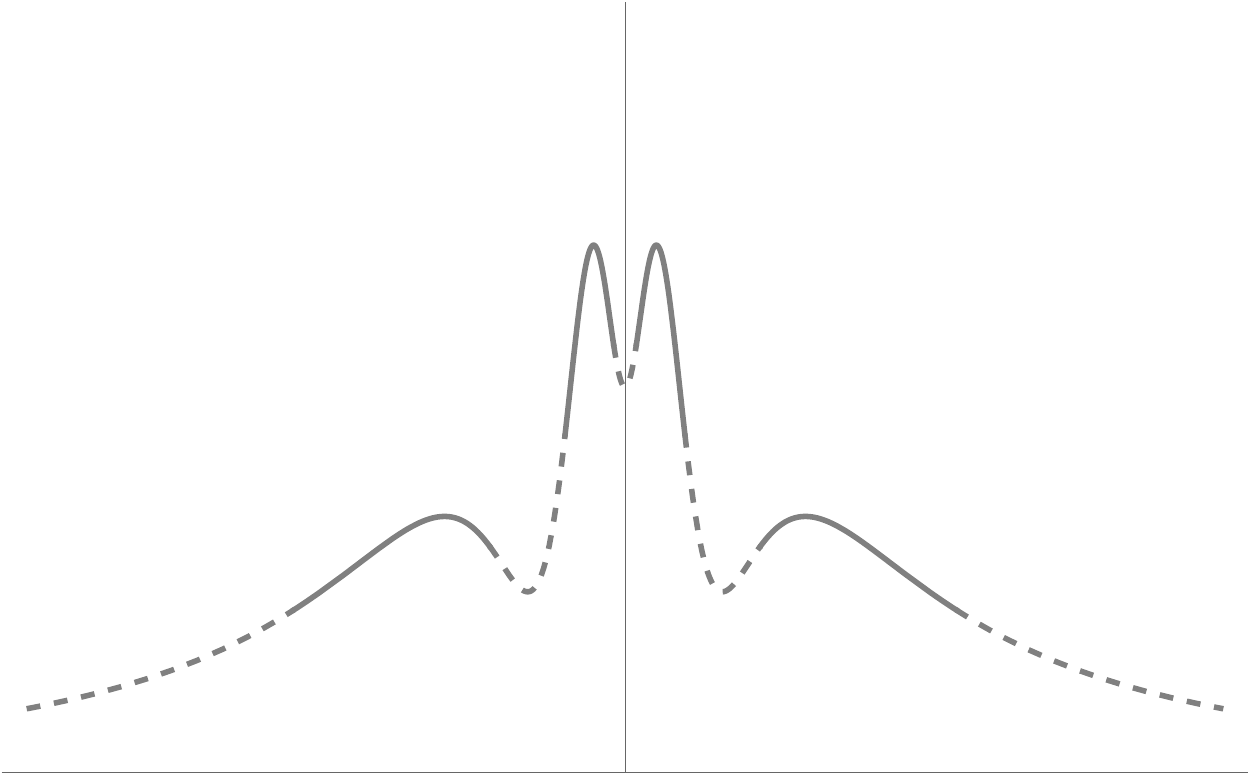} 
\includegraphics[width=0.43\textwidth,  trim=-20pt -30pt -20pt -20pt ]{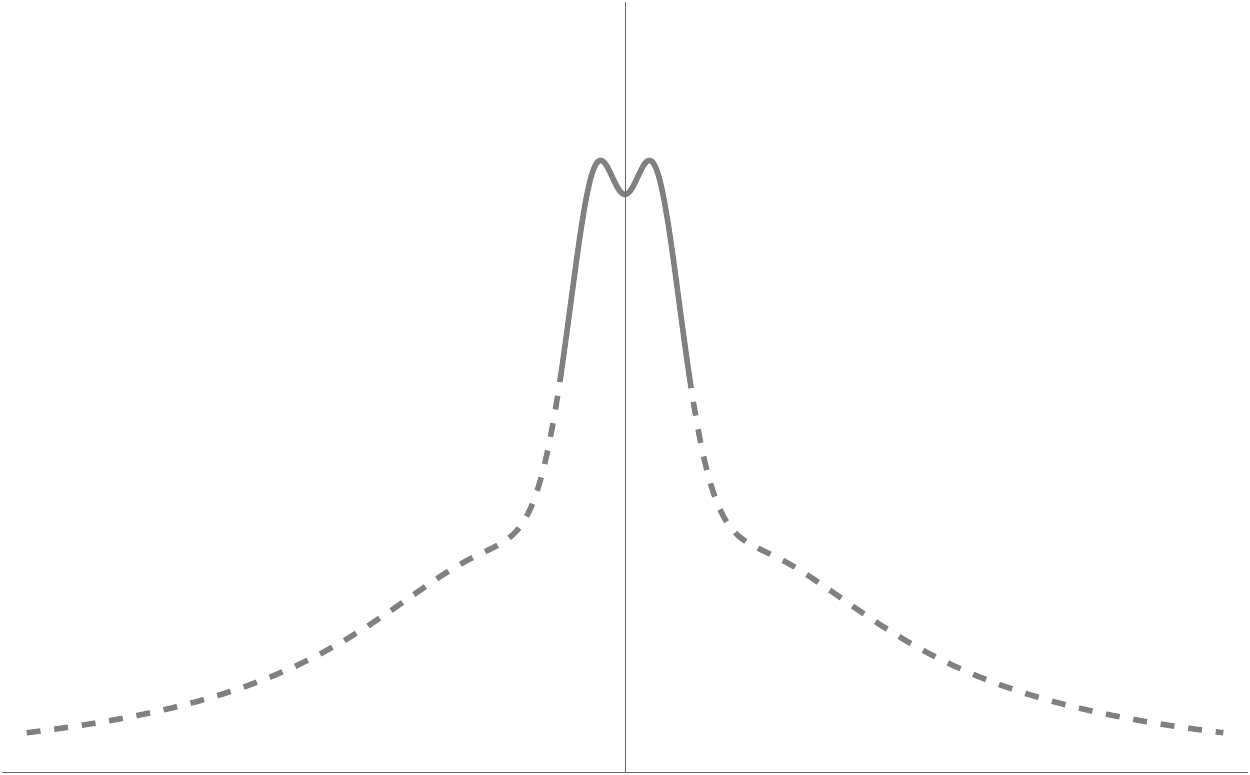} 
\end{tabular}
}

\caption{Graphs of~$\kappa_{3,a}$ with~$a=0$ (on top),~$a=\frac12$ (on the left) and~$a=\frac32$ (on the right). Regions with~$(-\Delta)^\frac12\kappa_{3,a}>0$ are plotted with a continuous line; regions with~$(-\Delta)^\frac12\kappa_{3,a}<0$ are plotted with a dashed line.}
\end{figure}

\end{example}




\bigbreak
It is interesting to compare the degree formula \eqref{degformula} we proved with some previous result on critical points of elliptic equations.

\medbreak
Theorem \ref{thmdeg} is consistent with a seminal result by Alessandrini on critical points of solutions to planar elliptic equations. When dealing with the Laplace equation set in the half-plane, Alessandrini's result reads as follows.

\begin{theorem}{\rm(\hspace{-0.003cm}\cite[Theorem 3.1]{A92})}
Assume that $\kappa \in C^2_{*,\beta}(\R)$ has $N\geq2$ global maxima, $N-1$ global minima and no other critical points. Suppose that $\lim_{x \to -\infty} \kappa(x) =  \lim_{x\to+\infty}\kappa(x)$ and that its value equals that of the global minima. Then, $\Gamma(\xi,\mu)$ has $K\ge1$ critical points, whose multiplicities $m_1,\dots,m_K$ satisfy
\begin{equation}\label{alessandrini}
\sum_{i=1}^Km_i=N-1.
\end{equation}
\end{theorem}
 
As already mentioned in Remark \ref{max-min}, if all maxima of $\kappa$ are global, then one gets $(-\Delta)^\frac12\kappa(\xi)>0$ at all such points $\xi$; similarly, if all minima of $\kappa$ are global, they satisfy $(-\Delta)^\frac12\kappa(\xi)<0$ at all such points $\xi$. Therefore, under the assumptions of Theorem \ref{thmdeg}, for $M^+_{\kappa}$ and $m^{+}_{\kappa}$ as in Theorem \ref{thmdeg}, one gets $M^+_\kappa=N$ and $m^+_\kappa=0$. Thus, Theorem \ref{thmdeg} shows that, for a sufficiently large $ R > 0$, $\deg(\nabla\Gamma,B_R \cap \R^2_{+},0)=1-N$.

As in the proof of Proposition \ref{propexactnum}, due the harmonicity of $\Gamma$ all its critical points are saddles. Hence, if all of them are non-degenerate, $\deg(\nabla\Gamma,B_R \cap \R^2_{+},0)=1-N$ equals minus the sum of the multiplicities of each critical point, namely we get formula \eqref{alessandrini}.

\bigbreak
We conclude the paper comparing our results with the ones in \cite{AL22}.

\medbreak
In the recent paper \cite{AL22}, Ahrend \& Lenzmann consider problem \eqref{eqk}, giving some results on the existence and uniqueness of solutions, as well as proving some qualitative properties of those solutions. In all of their main results they assume to prescribe a function $K(x)$ which is decreasing for positive $x$ and increasing for negative $x$, namely it satisfies $xK'(x)\le0$. Hence, the cases they cover are disjoint from the ones we consider in Theorem \ref{main2}. Actually, under the assumptions of \cite{AL22}, $K(x)$ has a unique global maximum in $x=0$ and so does $\kappa(x)=\epsilon^{-1}(K(x)-1)$. Thus, we have~$M^+_\kappa=1$ and~$m^+_\kappa=0$, and Theorem~\ref{thmdeg} gives no results.

In the case where $\pm \kappa$ is non-decreasing for negative $x$ and non-increasing for positive $x$, not only Theorem~\ref{thmdeg} does not provide solutions to \eqref{maineq}, but we get the following non-existence result.

\begin{proposition} Assume that $\kappa \in C^{1,\alpha}(\R) \cap L^{\infty}(\R)$ for some $\alpha > 0$ satisfies $|\kappa'(x)| \leq C (1+|x|)^{-\beta}$ for some constants $C > 0$ and $\beta > 0$.  If \eqref{maineq} has a solution $u$ in the form \eqref{u=U+phi} for some $\phi\in L^\infty(\R)$, 
$$
\int_\R x\kappa'(x)e^{u(x)}\,dx=0.
$$ 
In particular, $x\kappa'(x)$ must change sign.
\end{proposition}

\begin{proof}
From \cite[Lemmas 2.3 and 2.4]{AL22}, we get that any solution to \eqref{maineq} satisfy
\begin{equation*}
u(x)=-\frac\Lambda\pi\log|x|+O(1), \quad \textup{ as } |x| \to \infty,
\end{equation*}
and
\begin{equation} \label{ahrendLenzmann}
\frac\Lambda{2\pi}(\Lambda-2\pi)=\int_\R x(1+\epsilon\kappa(x))'e^{u(x)}\,dx=\epsilon\int_\R x\kappa'(x)e^{u(x)}\,dx,
\end{equation}
where 
$$
\Lambda:=\int_\R (1+\epsilon \kappa(x)) e^{u(x)}\,dx.
$$
Since $\phi$ is bounded, the (logarithmic) asymptotic behaviour of $u$ is the same than the one of $\mathcal U_{\mu,\xi}$ and thus, necessarily, $\Lambda=2\pi$. Hence, the result immediately follows from \eqref{ahrendLenzmann}. 
\end{proof}

\appendix

\section{An alternative proof of point~\ref{a} in Lemma~\ref{denslem}} \label{tranapp}

\noindent
We include here a different proof of point~\ref{a} of Lemma~\ref{denslem}, based on the generalization of an abstract transversality result due to Saut \& Temam. We start by recalling its statement, following Henry~\cite{H05}.

\begin{theorem}[{\!\!\cite[Theorem~5.4]{H05}}] \label{tran}
Let~$X, Y, Z$ be Banach spaces and~$U \subset X$,~$V \subset Y$ be open subsets. Let~$F: U \times V \to Z$ be a~$C^k$ map, for some integer~$k \ge 1$, and~$z_0 \in Z$. For all~$(x_0, y_0) \in F^{-1}(\{z_0\})$, assume that:
\begin{enumerate}[label=$(\roman*)$]
\item \label{i} The partial derivative~$\partial_x F(x_0, y_0): X \to Z$ is a Fredholm operator of index~$l < \alpha$;
\item \label{ii} The total derivative~$D F(x_0, y_0) = \langle \left( \partial_x F(x_0, y_0), \partial_y F(x_0, y_0) \right), \cdot \rangle : X \times Y \to Z$ is onto.
\end{enumerate}
Furthermore, suppose that:
\begin{enumerate}[label=$(\roman*)$]
\setcounter{enumi}{2}
\item \label{iii} The projection map~$F^{-1}(\{z_0\}) \to Y$ is~$\sigma$-proper, in the sense that~$F^{-1}(\{z_0\}) = \bigcup_{j \in \N} M_j$ with~$M_j \subset U \times V$ such that the projection~$M_j \ni (x, y) \mapsto y \in Y$ is proper for every~$j \in \N$.
\end{enumerate}
Then, the set~$\big\{ y \in V : z_0 \mbox{ is a critical value of } F(\cdot, y) \big\}$ is meager in~$Y$.
\end{theorem}

\begin{remark} We recall that a map is called \emph{proper} if preimages of compact subsets are compact, whereas a subset of a topological space~$Y$ is called \emph{meager} if it is the countable union of sets whose closures have empty interiors.
\end{remark}


Thanks to this theorem, we can move on to give a second proof of our genericity result.


\begin{proof}[Proof of Lemma~\ref{denslem} \ref{a}]
We shall apply Theorem~\ref{tran} to the~$C^1$ map~$F: U \times V \to Z$ defined by
$$
F(\xi,\mu,\kappa):=\nabla \Gamma(\xi,\mu)= \frac{1}{\pi} \left( \int_{\R} \frac{\kappa'(\xi+\mu y)}{1+y^2} \, dy, \int_{\R} \frac{\kappa'(\xi+\mu y)y}{1+y^2} \, dy\right),
$$
with~$X = Z = \R^2$,~$U = \R^2_+$,~$Y = V = C^2_{*, \beta}(\R)$, and~$z_0 = 0$. In order to do so, we need to check that assumptions~\ref{i},~\ref{ii}, and~\ref{iii} are verified.

Since~$X = Z$ is a finite dimensional space,~$\partial_x F$ is clearly a Fredholm operator of index~$0$. Hence,~\ref{i} holds true. Property~\ref{iii} is also satisfied, as can be checked by taking
$$
M_j := F^{-1}(\{ 0 \}) \cap \left( \left\{ (\xi, \mu) \in \overline{B_j} : \mu \ge \frac{1}{j} \right\} \times C_{*,\beta}^2(\R) \right),
$$
for every~$j \in \N$. It thus only remains to verify~\ref{ii}, i.e.,~that the linear map 
$$
(\xi,\mu,\kappa)\mapsto \partial_{(\xi,\mu)}F(\xi_0,\mu_0,\kappa_0)[\xi,\mu]+ \partial_{\kappa}F(\xi_0,\mu_0,\kappa_0)[\kappa],\qquad 
(\xi,\mu,\kappa)\in \mathbb R^2 \times C^{2}_{*,\beta}(\mathbb R),
$$
is onto whenever~$F(\xi_0,\mu_0,\kappa_0)=0$. Note that it is enough to prove that the linear map
$$
\kappa \mapsto \partial_{\kappa} F(\xi_0,\mu_0,\kappa_0)[\kappa] \in\mathbb R^2,
$$
is onto. In particular, as~$F$ is linear in~$\kappa$, the proof is reduced to finding two functions~$\kappa_1,\kappa_2\in C^{2}_{*,\beta}(\mathbb R)$ such that
\begin{equation}\label{trans1}
\partial_{\kappa}F(\xi_0,\mu_0,\kappa_0)[\kappa_1]=(1,0)\quad \hbox{ and } \quad  \partial_{\kappa}F(\xi_0,\mu_0,\kappa_0)[\kappa_2]=(0,1).
\end{equation}
For~$a_1, a_2 > 0$, we set
$$
h_1(y):= \frac{a_1}{1+y^6} \quad \hbox{and} \quad h_2(y):= \frac{a_2 \, y}{1+y^6}.
$$
Clearly, we can adjust~$a_1$ and~$a_2$ in a way that
$$
\int_{\mathbb R} \frac{h_1(y)}{1+y^2} \, dy = \pi,\quad \int_{\mathbb R} \frac{h_1(y)y}{1+y^2} \, dy=0, \quad  \int_{\mathbb R} \frac{h_2(y)}{1+y^2} \, dy=0,\quad \hbox{and}\quad \int_{\mathbb R} \frac{h_2(y)y}{1+y^2} \, dy = \pi.
$$
Then, if~$k_i$ denotes any primitive of~$h_i$, the functions~$\kappa_i(x):=\mu_0 k_i \! \left( \frac{x-\xi_0}{\mu_0}\right)$ belong to the space~$C^{2}_{*,\beta}(\R)$ and satisfy~\eqref{trans1}.

Thanks to Theorem~\ref{tran}, we infer that the complement of the set~$\mathscr{A}$ is meager in~$C^2_{*, \beta}(\R)$. This being a Banach space, we conclude that~$\mathscr{A}$ is dense in~$C^2_{*, \beta}(\R)$---the complement of a meager set in a complete metric space is dense, as a consequence of Baire's category theorem.
\end{proof}

\bibliography{biblio}
\bibliographystyle{abbrv}

\end{document}